\newif\ifmarek
\marektrue

\documentclass[reqno,twoside,11pt]{amsart}

\usepackage{srcltx}
\usepackage{amsmath}
\usepackage{amsfonts}
\usepackage{amssymb}
\usepackage{verbatim}
\usepackage{epsfig}
\usepackage{color}

\IfFileExists{epsf.def}{\input epsf.def}{\usepackage{epsf}}

\DeclareFontFamily{OT1}{eusb}{} \DeclareFontShape{OT1}{eusb}{m}{n} {<5> <6> <7> <8> <9> <10> <11> <12> <14.4> eusb10}{}
\DeclareMathAlphabet{\eusb}{OT1}{eusb}{m}{n}

\DeclareFontFamily{OT1}{eusm}{} \DeclareFontShape{OT1}{eusm}{m}{n} {<5> <6> <7> <8> <9> <10> <11> <12> <14.4> eusm10}{}
\DeclareMathAlphabet{\eusm}{OT1}{eusm}{m}{n}

\DeclareFontFamily{OT1}{eufm}{} \DeclareFontShape{OT1}{eufm}{m}{n} {<5> <6> <7> <8> <9> <10> <11> <12> <14.4> eufm10}{}
\DeclareMathAlphabet{\mathfrak}{OT1}{eufm}{m}{n}

\DeclareFontFamily{OT1}{fraktura}{}
\DeclareFontShape{OT1}{fraktura}{m}{n} {<5> <6> <7> <8> <9> <10> <11> <12> <13> <14.4> [1.1] eufm10}{}
\DeclareMathAlphabet{\fraktura}{OT1}{fraktura}{m}{n}

\DeclareFontFamily{OT1}{cmfi}{} \DeclareFontShape{OT1}{cmfi}{m}{n} {<5> <6> <7> <8> <9> <10> <11> <12> <13> <14.4> [0.9] cmfi10}{}
\DeclareMathAlphabet{\cmfi}{OT1}{cmfi}{b}{n}

\DeclareFontFamily{OT1}{cmss}{} \DeclareFontShape{OT1}{cmss}{m}{n} {<5> <6> <7> <8> <9> <10> <11> <12> <13> <14.4> cmss10}{}
\DeclareMathAlphabet{\cmss}{OT1}{cmss}{m}{n}

\IfFileExists{mathptmx.sty}{\usepackage{mathptmx}\usepackage{mathrsfs}}
{\usepackage{times}\usepackage{mathrsfs}}
\renewcommand{\mathcal}{\eusm}

\newtheoremstyle{thm}{1.5ex}{1.5ex}{\itshape\rmfamily}{} {\bfseries\rmfamily}{}{2ex}{}

\newtheoremstyle{def}{1.5ex}{1.5ex}{\slshape\rmfamily}{} {\bfseries\rmfamily}{}{2ex}{}

\newtheoremstyle{rem}{1.5ex}{1.5ex}{\rmfamily}{} {\bfseries\rmfamily}{} {1.5ex}{}

\newenvironment{proofsect}[1] {\vskip0.1cm\noindent{\rmfamily\itshape#1.}}{\qed\vspace{0.15cm}}

\theoremstyle{thm}
\newtheorem{theorem}{Theorem}[section]
\newtheorem{lemma}[theorem]{Lemma}
\newtheorem{proposition}[theorem]{Proposition}

\newtheorem*{Main Theorem}{Main Theorem.}
\newtheorem{corollary}[theorem]{Corollary}
\newtheorem*{special theorem}{Lindeberg-Feller Theorem for Martingales}

\newtheorem{assumption}[theorem]{Assumption}

\theoremstyle{def}

\theoremstyle{rem}
\newtheorem{remark}[theorem]{{\bfseries Remark}}

\numberwithin{equation}{section}

\renewcommand{\section}{\secdef\sct\sect}
\newcommand{\sct}[2][default]{%
\refstepcounter{section}
\addcontentsline{toc}{section}{{\tocsection {}{\thesection}{\!\!\!\!#1\dotfill}}{}}
\vspace{0.7cm}
\centerline{\scshape\thesection.\ #1} \nopagebreak \vspace{0.2cm}}
\newcommand{\sect}[1]{%
\vspace{0.4cm} \centerline{\large\scshape\rmfamily #1}
\vspace{0.2cm}}

\renewcommand{\subsection}{\secdef\subsct\sbsect}
\newcommand{\subsct}[2][default]{\refstepcounter{subsection}
\addcontentsline{toc}{subsection}
{{\tocsection{\!\!}{\hspace{1.2em}\thesubsection}{\!\!\!\!#1\dotfill}}{}}
\nopagebreak\vspace{0.75\baselineskip} {\flushleft\bf
\thesubsection~\bf #1.~}
\\*[3mm]\noindent
\nopagebreak}
\newcommand{\sbsect}[1]{\nopagebreak\vspace{0.75\baselineskip}\noindent
\textbf{#1.~}\\*[3mm]\noindent
\nopagebreak}

\renewcommand{\subsubsection}{%
\secdef \subsubsect\sbsbsect}
\newcommand{\subsubsect}[2][default]{%
\refstepcounter{subsubsection} 
\addcontentsline{toc}{subsubsection}{{\tocsection{\!\!}
{\hspace{3.05em}\thesubsubsection}{\!\!\!\!#1\dotfill}}{}}
\nopagebreak
\vspace{0.15\baselineskip} \nopagebreak {\flushleft\rmfamily
\itshape\thesubsubsection
\ \rmfamily #1\/.}\ }
\newcommand{\sbsbsect}[1]{\vspace{0.1cm}\noindent
\rmfamily \itshape
\arabic{section}.\arabic{subsection}.\arabic{subsubsection} \
\sffamily #1\/.\ }

\renewcommand{\caption}[1]{%
\vglue0.5cm
\refstepcounter{figure}
\begin{minipage}{0.9\textwidth}\small {\sc Figure~\thefigure. }#1\end{minipage}}

\newcommand{\dist}{\operatorname{dist}}
\newcommand{\supp}{\operatorname{supp}}

\newcommand{\textd}{\text{\rm d}\mkern0.5mu}
\newcommand{\texti}{\text{\rm  i}\mkern0.7mu}
\newcommand{\texte}{\text{\rm e}}

\newcommand{\1}{\operatorname{\sf 1}\!}

\newcommand{\CC}{\mathcal C}

\newcommand{\FF}{\mathcal F}

\newcommand{\HH}{\mathcal H}

\newcommand{\NN}{\mathcal N}

\newcommand{\B}{\mathbb B}

\newcommand{\E}{\mathbb E}

\newcommand{\N}{\mathbb N}

\newcommand{\BbbP}{\mathbb P}

\newcommand{\R}{\mathbb R}

\newcommand{\T}{\mathbb T}

\newcommand{\Z}{\mathbb Z}

\newcommand{\twoeqref}[2]{(\ref{#1}--\ref{#2})}
\newcommand{\cc}{{\text{\rm c}}}

\def\myffrac#1#2 in #3{\raise 2.6pt\hbox{$#3 #1$}\mkern-1.5mu\raise 0.8pt\hbox{$#3/$}\mkern-1.1mu\lower 1.5pt\hbox{$#3 #2$}}
\newcommand{\ffrac}[2]{\mathchoice%
{\myffrac{#1}{#2} in \scriptstyle}
{\myffrac{#1}{#2} in \scriptstyle}
{\myffrac{#1}{#2} in \scriptscriptstyle}
{\myffrac{#1}{#2} in \scriptscriptstyle}
}

\newcommand{\hate}{\hat{\text{\rm e}}}

\newcommand{\ssup}[1] {{\scriptscriptstyle{({#1}})}}

\newcommand{\wt}{\widetilde}

\begin{document}


\title[Eigenvalue fluctuations\hfill\qquad]{Eigenvalue fluctuations for lattice Anderson Hamiltonians}
\author[\qquad \hfill Biskup, Fukushima, K\"onig]{Marek Biskup$^{1}$,\, Ryoki Fukushima$^{2}$\and\,\,Wolfgang K\"onig$^{3,4}$}

\thanks{\hglue-4.5mm\fontsize{9.6}{9.6}\selectfont\copyright\,2016 by M.~Biskup, R.~Fukushima and W.~K\"onig. Reproduction, by any means, of the entire article for non-commercial purposes is permitted without charge.
\vspace{2mm}
}

\maketitle

\vglue-2mm

\centerline{\textit{$^1$Department of Mathematics, UCLA, Los Angeles, California, USA}}
\centerline{\textit{$^2$Research Institute in Mathematical Sciences, Kyoto University, Kyoto, Japan}}
\centerline{\textit{$^3$Weierstra\ss-Institut f\"ur Angewandte Analysis und Stochastik, Berlin, Germany}}
\centerline{\textit{$^4$Institut f\"ur Mathematik, Technische Universit\"at Berlin, Berlin, Germany}}

\vglue3mm


\begin{abstract}
We study the statistics of Dirichlet eigenvalues of the random Schr\"odinger operator $-\epsilon^{-2}\Delta^{(\textd)}+\xi^{(\epsilon)}(x)$, with~$\Delta^{(\textd)}$ the  discrete Laplacian on~$\Z^d$ and $\xi^{(\epsilon)}(x)$ uniformly bounded independent random variables, on sets of the form $D_\epsilon:=\{x\in\Z^d\colon x\epsilon\in D\}$ for~$D\subset\R^d$ bounded, open and with a smooth boundary. If $\E\xi^{(\epsilon)}(x)=U(x\epsilon)$ holds for some bounded and continuous $U\colon D\to\R$, we show that, as $\epsilon\downarrow0$,
the $k$-th eigenvalue converges to the $k$-th Dirichlet eigenvalue of 
the homogenized operator $-\Delta+U(x)$, where~$\Delta$ is the continuum Dirichlet Laplacian on~$D$.
Assuming further that $\text{\rm Var}(\xi^{(\epsilon)}(x))=V(x\epsilon)$ for some positive and
continuous $V\colon D\to\R$, we establish a multivariate central limit
theorem for simple eigenvalues centered by their expectation. The limiting covariance for a given pair of simple eigenvalues is expressed as an integral
of~$V$ against the product of squares of the corresponding eigenfunctions of~$-\Delta+U(x)$.
\end{abstract}

\section{Introduction}
\vglue-2mm\subsection{The model and main results}
\noindent
The phenomenological description of physical processes such as heat or electric conductivity in materials is typically governed by differential equations with smoothly varying coefficients. However, due to an underlying crystalline structure as well as presence of impurities, the physical characteristics of materials change quite rapidly at the microscopic level. The apparent discrepancy in assumed regularity is reconciled mathematically by homogenization theory which provides tools to integrate out fine-scale oscillations and extract, in specific cases, a suitable continuum limit. A key point for modeling is to track how the microscopic details express into the values of material constants.

In this article we take up a study of one specific example of this approach. The general context is the spectral side of stochastic homogenization, which is currently a highly active research area. The quantities of our interest are low-lying eigenvalues of random Schr\"odinger operators called Anderson Hamiltonians. Such operators naturally appear in theories of disordered materials in solid state physics; indeed, they describe the motion of a single electron through a crystal with impurities. Our focus will be on the limiting statistics of these low-lying eigenvalues with the aim to capture both the leading-order behavior, which turns out to be deterministic by a Law of Large Numbers, as well as the leading-order random term, which turns out to be Gaussian by a Central Limit Theorem. Asymptotic expansions for eigenvalues of such operators are relevant for various natural questions of interest (e.g., decay of the heat kernel) as well as for numerical analysis of such systems.  Some additional motivation for our work will be described in Section~\ref{sec2}.

Let us move to precise definitions and results. Let~$D$ be a bounded open subset of~$\R^d$ whose boundary is~$C^{1,\alpha}$ for some~$\alpha>0$. Given an $\epsilon>0$, we define the discretized version of~$D$ as
\begin{equation}
\label{E:1.1}
D_\epsilon:=\bigl\{x\in\Z^d\colon \dist_\infty(x\epsilon,D^\cc)>\epsilon\bigr\}
\end{equation}
where $\dist_\infty$ is the $\ell^\infty$-distance in~$\R^d$ and~$D^\cc$ is the complement of~$D$.
For any numbers $\xi^{(\epsilon)}(x)$, $ x\in D_\epsilon$, define an operator (a matrix) $H_{D_\epsilon,\xi}$ acting on the linear space of functions $f\colon D_\epsilon\to\R$ that vanish outside~$D_\epsilon$ (i.e., the Dirichlet boundary condition is imposed) via
\begin{equation}
\label{E:1.2}
(H_{D_\epsilon,\xi}f)(x):=-\epsilon^{-2}(\Delta^{(\textd)} f)(x)+\xi^{(\epsilon)}(x)f(x),
\end{equation}
where $\Delta^{(\textd)}$ is the standard lattice Laplacian
\begin{equation}
(\Delta^{(\textd)} f)(x):=\sum_{y\colon |y-x|=1}\bigl[f(y)-f(x)\bigr]
\end{equation}
with $|x|$ denoting the $\ell^{1}$-norm of~$x$.
The operator $H_{D_\epsilon,\xi}$ is an example of the Anderson Hamiltonian. 
Note that, by scaling the spatial coordinates by~$\epsilon$, one can equivalently regard $H_{D_\epsilon,\xi}$ as an operator on functions on $\epsilon D_\epsilon$. The kinetic term, $\epsilon^{-2}\Delta^{(\textd)}$, is a natural approximation of the continuous Laplacian on $D$.

The potential $\xi^{(\epsilon)}$ will be taken random with values at different vertices independent of each other. Although this means that $\xi^{(\epsilon)}$ will be quite rough in each specific realization, we will require, as is common in homogenization theory, that the probability laws of individual $\xi^{(\epsilon)}(x)$ vary continuously with the position. Namely, all results in this note will be based on the following assumptions:

\begin{assumption}
\label{ass1}
There are numbers $a,b\in\R$ with $a<b$ and bounded continuous functions $U\colon D\to\R$ and $V\colon D\to(0,\infty)$ such that the following holds for each~$\epsilon>0$:
\settowidth{\leftmargini}{(11)}
\begin{enumerate}
\item[(1)] the random variables $\{\xi^{(\epsilon)}(x)\colon x\in D_\epsilon\}$, are independent, 
\item[(2)] for any~$x\in D_\epsilon$,
\begin{equation}
\label{E-bounded}
a\le \xi^{(\epsilon)}(x)\le b,
\end{equation}
\item[(3)] for any~$x\in D_\epsilon$,
\begin{equation}
\label{E:1.4a}
\E\xi^{(\epsilon)}(x)=U(x\epsilon)\quad\text{and}\quad \text{\rm Var}\bigl(\xi^{(\epsilon)}(x)\bigr)=V(x\epsilon).
\end{equation}
\end{enumerate}
\end{assumption}

We will write $\BbbP_\epsilon$ to denote the law of $\xi^{(\epsilon)}$ but will not mark the $\epsilon$-dependence explicitly on expectation.
To ease our notations, we will also often omit marking the $\epsilon$-dependence of~$\xi$. The boundedness assumption \eqref{E-bounded} can be relaxed somewhat but we refrain from doing so in order to keep the paper focused on the phenomena we wish to describe. Also, most of our result apply even when the equalities \eqref{E:1.4a} just hold in the limit~$\epsilon\downarrow0$. 

\smallskip As already stated, our focus will be on the asymptotic behavior of the low-lying part of the spectrum of $H_{D_\epsilon,\xi}$ in the limit as~$\epsilon\downarrow0$. Here we note that, since $H_{D_\epsilon,\xi}$ is a symmetric $|D_\epsilon|\times|D_\epsilon|$-matrix, its eigenvalues are all real-valued and can be ordered as
\begin{equation}
\lambda^{\ssup 1}_{D_\epsilon,\xi}\le \lambda^{\ssup 2}_{D_\epsilon,\xi}\le\dots\le \lambda^{\ssup {|D_\epsilon|}}_{D_\epsilon,\xi}.
\end{equation}
As our first result we note that, in the limit~$\epsilon\downarrow0$, these converge to the eigenvalues of a suitable (homogenized) continuum operator:

\begin{theorem}
\label{thm1.1}
Under Assumption~\ref{ass1}, for each~$k\ge1$,
\begin{equation}
\label{E:1.4}
\lambda^{\ssup k}_{D_\epsilon,\xi}\,\underset{\epsilon\downarrow0}{\overset{}{\longrightarrow}}\,\lambda_D^{\ssup k} \qquad\text{\rm in probability},
\end{equation}
where $\lambda_D^{\ssup k}$ is the $k$-th smallest eigenvalue of the operator $-\Delta+U(x)$ on~$\cmss H^1_0(D)$, with~$\Delta$ {denoting} the continuum Laplacian. 
\end{theorem}

{Here, as} usual, $\cmss H^1_0(D)$ denotes the closure of the set of infinitely differentiable and compactly supported functions in~$D$ with respect to the norm $\Vert f\Vert_{\cmss H^1(D)}:=(\Vert f\Vert_{L^2(D)}^2+\Vert\nabla f\Vert_{L^2(D)}^2)^{1/2}$. Thanks to our conditions on~$D$ and~$U$, the spectrum of~$-\Delta+U(x)$ is discrete with no eigenvalue more than finitely degenerate. Moreover, any orthonormal basis of eigenfunctions~$\varphi_D^{\ssup k}$ consists of functions that are continuously differentiable on~$\overline D$. See Lemma~\ref{lemma-regularity} for details.

Statements of the form \eqref{E:1.4} have been proved in various contexts before; see, e.g., the monograph of Jikov, Kozlov and Oleinik~\cite{ZKO} and further discussion in Section~\ref{sec2}. However, concerning the eigenvalues of Anderson Hamiltonians, we have found only one homogenization result due to Bal~\cite{Bal08}, which is moreover restricted to~$d\le 3$. See Section~\ref{Bal} below for more details.

\smallskip
The formula \eqref{E:1.4} gives the leading-order deterministic behavior of the spectrum of~$H_{D_\epsilon,\xi}$. Naturally, one might be interested in the subleading terms or even a full asymptotic expansion in powers of~$\epsilon$. Some of the terms in this expansion are likely to be deterministic --- e.g., those describing the boundary effects --- while others could genuinely be random. The leading order {random} term captures the fluctuations of the eigenvalues around their mean. To understand the typical scale of such fluctuations, we note the following concentration estimate:

\begin{theorem}
\label{thm-1.2a}
 Under Assumption~\ref{ass1}, for each~$k\ge1$, there is $c>0$ such that for all~$t>0$ and all~$\epsilon\in(0,1)$,
\begin{equation}
\label{E:1.6}
\BbbP_\epsilon\Bigl(\bigl|\lambda^{\ssup
k}_{D_\epsilon,\xi}-\E\lambda^{\ssup
k}_{D_\epsilon,\xi}\bigr|>t\Bigr)\le 4\texte^{-c\,t^2\epsilon^{-d}}.
\end{equation}
\end{theorem}

If~$c(k)$ marks the largest~$c$ for which \eqref{E:1.6} holds, our proof gives $c(k)\gtrsim k^{-2}\texte^{-2\lambda_D^{\ssup k}}$. However, this is probably quite far from optimal. Still, thanks to \eqref{E:1.6} the random variables
\begin{equation} 
\frac{\lambda^{\ssup k}_{D_\epsilon,\xi}-\E\lambda^{\ssup k}_{D_\epsilon,\xi}}{\epsilon^{d/2}}
\end{equation}
are tight in the limit~$\epsilon\downarrow0$ and, in fact, have uniform Gaussian tails. This suggests a possible Gaussian limit theorem. And indeed, as our next and also main result shows, a Central Limit Theorem (CLT) holds and that so jointly for the collection of all eigenvalues that are simple in the limit $\epsilon\downarrow0$:

\begin{theorem}
\label{thm1.2}
Suppose Assumption~\ref{ass1} holds, fix $n\in\N$ and let $k_1,\dots,k_n\in\N$ be distinct indices such that the Dirichlet eigenvalues $\lambda_D^{\ssup{k_1}},\dots,\lambda_D^{\ssup{k_n}}$ of $-\Delta+U(x)$ on~$D$ are simple. Then, in the limit as $\epsilon\downarrow0$, the law of the random vector
\begin{equation}
\Biggl(\frac{\lambda_{D_\epsilon,\xi}^{\ssup{k_1}}-\E \lambda_{D_\epsilon,\xi}^{\ssup{k_1}}}{\epsilon^{d/2}},\dots,\frac{\lambda_{D_\epsilon,\xi}^{\ssup{k_n}}-\E \lambda_{D_\epsilon,\xi}^{\ssup{k_n}}}{\epsilon^{d/2}}\Biggr)
\end{equation}
tends weakly to a multivariate normal with mean zero and covariance matrix $\sigma_D^2=\{\sigma_{ij}^2\}_{i,j=1}^n$ that is given by
\begin{equation}
\label{cov}
\sigma^2_{ij}:=\int_D\bigl|\varphi_D^{\ssup{k_i}}(x)\bigr|^2\bigl|\varphi_D^{\ssup{k_j}}(x)\bigr|^2\,V(x)\,\textd x,
\end{equation}
{where} $\varphi_D^{\ssup i}$ denotes the $i$-th normalized eigenfunction of~$-\Delta+U(x)$ and~$V(x)$ is as in Assumption~\ref{ass1}(3).
\end{theorem}

Results of this kind are only few and far in-between. One context where such a limit law has been claimed is the \emph{crushed ice problem}; see Section~\ref{sec2.1} for further discussion and references. Understanding the crushed-ice problem has in fact been a prime motivation for this work. 
We note that the aforementioned paper~\cite{Bal08} also contains a Gaussian fluctuation result but again only for $d\le 3$; see Section~\ref{Bal}.
\begin{remark}

This remark concerns the restriction of Theorem~\ref{thm1.2} to simple eigenvalues. 
It is clear that some restriction is needed whenever the expectations of two eigenvalues fall within $o(\epsilon^{d/2})$ of each other. Although, by Theorem~\ref{thm-1.2a}, the fluctuations of individual eigenvalues perhaps remain CLT-like, under degeneracy they \emph{decide} the order and hence no Gaussian limit is possible. The precise ordering also depends on their expectations and so further control of subleading terms in \eqref{E:1.4} would be required in order to make a meaningful conclusion in the end. (Of course, alternative formulations may still be possible --- e.g., in terms of the Green operator or spectral density --- but our present proofs would not apply anyway.)
\end{remark}

\subsection{Key underlying idea}
From the perspective of the theory of random Schr\"odinger operators it is interesting to ponder about where the {principal contribution to the} fluctuations of the eigenvalues {comes} from. Our method of proof indicates this quite clearly. Let $g_{D_{\epsilon},\xi}^{\ssup k}$ henceforth denote any eigenfunction of~$H_{D_\epsilon,\xi}$ for the eigenvalue~$\lambda_{D_\epsilon,\xi}^{\ssup k}$ normalized so that
\begin{equation}
\label{E:3.1}
\sum_{x\in D_\epsilon}\bigl|g_{D_{\epsilon},\xi}^{\ssup k}(x)\bigr|^2=1.
\end{equation}
Let $C^{\ssup k}_\epsilon$ denote the event that $\lambda^{\ssup k}_{D_\epsilon,\xi}$ is non-degenerate and note that, by \eqref{E:1.4}, $\BbbP_\epsilon(C^{\ssup k}_\epsilon)\to1$ as~$\epsilon\downarrow0$ for any~$k$ such that the Dirichlet eigenvalue~$\lambda^{\ssup k}_D$ of $-\Delta+U(x)$ is non-degenerate, i.e., simple. On~$C^{\ssup k}_\epsilon$, write
\begin{equation}
T_{D_\epsilon,\xi}^{\ssup k}:=\sum_{x\in\Z^d}\epsilon^{-2}\bigl|\nabla^{(\textd)}g_{D_\epsilon,\xi}^{\ssup k}(x)\bigr|^2,
\end{equation}
to denote the \emph{kinetic energy} associated with the $k$-th
eigenspace of $H_{D_\epsilon,\xi}$, {where} $\nabla^{(\textd)}f(x)$ is the
vector whose~$i$-th component is $f(x+\hate_i)-f(x)$, for~$\hate_i$
denoting the~$i$-th unit vector in~$\R^d$. We
regard~$g_{D_{\epsilon},\xi}^{\ssup k}$ as extended by zero to all
of~$\Z^d$. By testing the eigen-equation by the eigenfunction and using the summation by parts, we get the following expression of the eigenvalue: 
\begin{equation}
\lambda_{D_\epsilon,\xi}^{\ssup k}=T_{D_\epsilon,\xi}^{\ssup k}+\sum_{x\in D_\epsilon}\xi(x)g_{D_\epsilon,\xi}^{\ssup k}(x)^2.
\end{equation}
The following theorem implies that the main fluctuation comes from the \emph{potential energy} part, i.e., the second term on the right hand side.

\begin{theorem}
\label{thm1.3}
Suppose Assumption~\ref{ass1} holds and that $\lambda^{\ssup k}_D$
 is simple. Then,
\begin{equation}
\label{E:1.10}
\epsilon^{-d}\text{\rm Var}\bigl(\,T_{D_\epsilon,\xi}^{\ssup k}\,\big|\,C_\epsilon^{\ssup k}\,\bigr)\,\underset{\epsilon\downarrow0}\longrightarrow\,0
\end{equation}
and
\begin{equation}
\label{E:1.11}
\epsilon^{-d}\sum_{x\in D_\epsilon}\text{\rm Var}\bigl(\,g_{D_\epsilon,\xi}^{\ssup k}(x)^2\,\big|\,C_\epsilon^{\ssup k}\,\bigr)\,\underset{\epsilon\downarrow0}\longrightarrow\,0.
\end{equation}
\end{theorem}

Based on \eqref{E:1.10}, the exact form of the covariance is easy to explain as well: just replace  $g_{D_\epsilon,\xi}^{\ssup k}(x)$ by the eigenfunction 
$\varphi_D^{\ssup k}$ of the limiting operator $-\Delta+U$ and note that the potential energy thus becomes a weighted  sum of i.i.d.\ random variables for which the central limit 
theorem with covariance~\eqref{cov} is well-known.

It turns out that an \emph{a priori} knowledge of {\twoeqref{E:1.10}{E:1.11}} is nearly enough to justify the central limit theorem in Theorem~\ref{thm1.2}. Indeed, let  $\E^{\ssup k}$ denote the conditional expectation given~$C_\epsilon^{\ssup k}$ and let us, for ease of notation, drop the subindices on $\lambda^{\ssup k}_{D_\epsilon,\xi}$, $T_{D_\epsilon,\xi}^{\ssup k}$ and $g_{D_\epsilon,\xi}^{\ssup k}$. On~$C_\epsilon^{\ssup k}$ we have
\begin{equation}
\lambda^{\ssup k}-
\E^{\ssup k} \lambda^{\ssup k}=T^{\ssup k}-\E^{\ssup k} T^{\ssup k}
+\sum_{x\in D_\epsilon}\Bigl(\xi(x)g^{\ssup k}(x)^2-\E^{\ssup k}\bigl(\xi(x)g^{\ssup k}(x)^2\bigr)\Bigr).
\end{equation}
The sum on the right can be recast as
\begin{multline}
\qquad
\sum_{x\in D_\epsilon}\bigl[\xi(x)-\E^{\ssup k}\xi(x)\bigr]\E^{\ssup k}\bigl(g^{\ssup k}(x)^2\bigr)+\sum_{x\in D_\epsilon}\xi(x)\bigl[g^{\ssup k}(x)^2-\E^{\ssup k}(g^{\ssup k}(x)^2)\bigr]
\\
+\sum_{x\in D_\epsilon}\E^{\ssup k}\Bigl(\bigl(\xi(x)-\E^{\ssup k}\xi(x)\bigr)\bigl(g^{\ssup k}(x)^2-\E^{\ssup k}(g^{\ssup k}(x)^2)\bigr)\Bigr).
\qquad
\end{multline}
A routine use of the Cauchy-Schwarz inequality shows that the second moment of the latter two sums is dominated by (powers of) the sum in \eqref{E:1.11}. Using also \eqref{E:1.10} we get
\begin{equation}
\lambda^{\ssup k}-
\E^{\ssup k} \lambda^{\ssup k}=o({\epsilon^{d/2}})+\sum_{x\in D_\epsilon}\bigl[\xi(x)-\E^{\ssup k}\xi(x)\bigr]\E^{\ssup k}\bigl(g^{\ssup k}(x)^2\bigr),
\end{equation}
{where $o(\epsilon^{d/2})$ represents a random variable whose 
variance is $o(\epsilon^d)$.}
Under the assumption that the $k$-th eigenvalue of~$-\Delta+U(x)$ is non-degenerate, the complement of $C_\epsilon^{\ssup k}$ can be covered by events from \eqref{E:1.6} for indices~$k-1$, $k$ and~$k+1$. This permits us to replace the conditional expectations of $\lambda^{\ssup k}$ and~$\xi(x)$ by unconditional ones. To get the multivariate CLT stated in Theorem~\ref{thm1.2}, it then suffices to show
\begin{equation}
{\epsilon^{-d}}\E^{\ssup k}\bigl(g^{\ssup k}(\lfloor \cdot/\epsilon\rfloor)^2\bigr)
\,\underset{\epsilon\downarrow0}\longrightarrow\,\bigl|\varphi_D^{\ssup k}(\cdot)\bigr|^2
\end{equation}
in $L^2(D,\textd x)$, for any~$k$ of interest.  As we will see, our proof of Theorems~\ref{thm1.2} and~\ref{thm1.3} is indeed strongly based on controlling the convergence of the discrete eigenfunctions to the continuous ones in proper $L^p$-norms. 

\begin{remark}
\label{rem-1.7}
As is common in homogenization theory, analyzing differential equations with rapidly varying coefficients typically requires separating the rapid oscillations into, or compensating for them by, a ``corrector'' term. The reader may thus be surprised to find that no such term needs to be introduced in our case. This is because this term is naturally of a smaller order in~$\epsilon$, and thus will not contribute to the fluctuations of the eigenvalues.

We can elucidate this further by invoking rank-one perturbation and \eqref{E:1.4}; see Proposition~\ref{prop3} and Lemma~\ref{unif-g}. Define~$\Psi^{\ssup k}$ by the equation
\begin{equation}
\epsilon^{-d/2}g_{D_\epsilon,\xi}^{\ssup k}(x)=\varphi_{D}^{\ssup k}(x\epsilon)+\epsilon^2\Psi^{\ssup k}(x).
\end{equation}
Invoking the eigenvalue equations, we then have
\begin{equation}
\begin{aligned}
\Delta^{(\textd)}\Psi^{\ssup k}(x)&=\epsilon^{-2-d/2}\Delta^{(\textd)}g_{D_\epsilon,\xi}^{\ssup k}(x)-\epsilon^{-2}\Delta^{(\textd)}\varphi_{D}^{\ssup k}(\cdot\,\epsilon)(x)
\\
&\approx \bigl(\lambda_{D_\epsilon,\xi}^{\ssup k}-\xi(x)\bigr)\epsilon^{-d/2}g_{D_\epsilon,\xi}^{\ssup k}(x)
-\bigl(\lambda_D^{\ssup k}-U(x\epsilon)\bigr)\varphi_{D}^{\ssup k}(x\epsilon),
\end{aligned}
\end{equation}
where we approximated the discrete Laplacian by its continuous counterpart. Assuming that $\epsilon^{-d/2}g_{D_\epsilon,\xi}^{\ssup k}(x)$ is in fact pointwise close to $\varphi_D^{\ssup k}(x\epsilon)$, we get
\begin{equation}
-\Delta^{(\textd)}\Psi^{\ssup k}(x)= \bigl(\xi(x)-U(x\epsilon)+o(1)\bigr)\varphi_{D}^{\ssup k}(x\epsilon),
\end{equation}
i.e., $\Psi^{\ssup k}$ solves a corrector-like Poisson equation. Since the Dirichlet Laplacian on~$D_\epsilon$ is invertible,~$\Psi^{\ssup k}$ can in principle be computed and studied. Dropping the $o(1)$-term suggests that $\Psi^{\ssup k}$ has finite variance in~$d\ge5$.
\end{remark}

\subsection{Outline}
The remainder of this paper is organized as follows: In the next section, we review some earlier work related to the present article. In Section~\ref{sec3} we establish Theorem~\ref{thm1.1} along with some useful regularity estimates on discrete and continuous eigenfunctions. In Section~\ref{sec4} we prove Theorem~\ref{thm-1.2a} dealing with concentration of the law of discrete eigenvalues. Then, in Section~\ref{sec5}, we proceed to prove our main result (Theorem~\ref{thm1.2}). Theorem~\ref{thm1.3} is then derived readily as well.

\section{Related work}
\label{sec2}\noindent
Before we delve into the proofs, let us make some connections to the existing literature. These have insofar been suppressed in order to keep the presentation focused. 

\subsection{Homogenization approach}
\label{Bal}\noindent
As alluded to earlier, a result closely related to ours has been derived by Bal~\cite{Bal08}. There the operator of the form $\mathcal{H}_{\epsilon,q}=-\Delta+{q(x/\epsilon)}$ in $D\subset\R^d$ with Dirichlet boundary condition is studied, where~$q$ is a random centered stationary field. Note that this can naturally be regarded as a spatially scaled version of our model. (Bal in fact studied the more general situation where $\Delta$ is replaced by a pseudo differential operator.) In dimensions $d\le 3$ and under the assumptions that
\begin{itemize}
 \item either $q$ is bounded and has an integrable correlation function, or 
 \item $\E[q(0)^6]<\infty$ and a mixing condition holds ([H2] on page 683 of~\cite{Bal08}), 
\end{itemize}
it is proved in Section 5.2 that the $k$-th smallest eigenvalue $\lambda_{\epsilon,q}^{\ssup k}$ of $\mathcal{H}_{\epsilon,q}$ has Gaussian fluctuations around $\lambda^{\ssup k}_D$ with $U\equiv 0$, provided this eigenvalue is simple. This is slightly different from our result, which shows a CLT around the expectation. In the case $d\le 3$, we \emph{a posteriori} know that $\E[\lambda_{\epsilon,q}^{\ssup k}]-\lambda^{\ssup k}_D=o(\epsilon^{d/2})$ by combining the result of Bal with ours, but we do not know how to prove this directly. 

The argument in~\cite{Bal08} is based on a perturbation expansion of the resolvent operator and an explicit representation of the leading-order \emph{local} correction to the eigenfunctions; cf.\ Remark~\ref{rem-1.7}. In order to control the remainder terms, one then needs that the Green function of the homogenized operator is square integrable, and this requires the restriction to $d\le 3$.
The method employed in the present article is different in it avoids having to deal with local perturbations altogether. Incidentally, as was recently shown by Gu and Mourrat~\cite{Gu-Mourrat}, for the random elliptic operators (see Subsection~\ref{RCM} below for a formulation) the limit laws of the local and global fluctuations to eigenfunctions are in fact not even the same.

\subsection{Crushed-ice problem}
\label{sec2.1}\noindent
Our attention to fluctuations of Dirichlet eigenvalues arose from our interest in the so called \emph{crushed ice} problem. This is a problem in the continuum where one considers a bounded open set~$D\subset\R^d$ with~$m$ Euclidean balls $B(x_1,\epsilon),\dots,B(x_m,\epsilon)$ of radius~$\epsilon$ removed from its interior. The positions $x_1,\dots,x_m$ of the centers of these balls are drawn independently from a common distribution $\rho(x)\,\textd x$ on~$D$. The principal question is how the eigenvalues of the Laplacian in
\begin{equation}
D_\epsilon:=D\smallsetminus(B(x_1,\epsilon)\cup\dots\cup B(x_m,\epsilon))
\end{equation}
behave in the limit as~$\epsilon\downarrow0$, for interesting choices of $m=m(\epsilon)\to\infty$. (The most natural boundary conditions are Neumann on~$\partial D$ and Dirichlet on $\partial B(x_i,\epsilon)$ but all mixtures of these can be considered.) 
To make the connection to our problem, note that one can view the negative Laplacian on~$D_\epsilon$ as the operator $-\Delta+\xi(x)$ on~$D$ with~$\xi(x)$ vanishing on~$D_\epsilon$ and $\xi(x)=\infty$ for~$x\in D\smallsetminus D_\epsilon$.

Since its introduction by Kac in 1974, much effort went into analyzing the crushed ice problem in various regimes of dependence of~$m$ on~$\epsilon$. The main references include Kac~\cite{Kac}, Huruslov and Marchenko~\cite{HM}, Rauch and Taylor~\cite{RT}; see also the monographs by Simon~\cite{Simon} and Sznitman~\cite{Sznitman}. More recently, extensions to non-homogeneous kinetic terms have also been considered, e.g., by Douanla~\cite{Douanla} and Ben-Ari~\cite{Ben-Ari}. The most interesting limit is obtained when
\begin{equation}
m(\epsilon)\,\text{Cap}\bigl(B(0,\epsilon)\bigr)\,\underset{\epsilon\downarrow0}\longrightarrow\,\mu\in(0,\infty),
\end{equation}
where $\text{Cap}(A)$ denotes the Newtonian capacity of~$A$
when $d\ge 3$ and the capacity for the operator~$-\Delta+1$ when
$d=2$. 
The $k$-th Dirichlet eigenvalue of~$-\Delta$ in~$D_\epsilon$
then tends to that of the Schr\"odinger operator~$-\Delta+\mu\rho(x)$
on~$D$. Note the appearance of a non-trivial ``potential'' $\mu\rho(x)$ despite the fact that the total volume occupied by the~$m$ balls vanishes in the stated limit.

The problem of fluctuations was in this context taken up by Figari, Orlandi and Teta~\cite{FOT} and later by Ozawa~\cite{Ozawa}. Both of these studies infer a (single-variate) Central Limit Theorem assuming simplicity of the limiting eigenvalue but they are confined to the case of~$d=3$. In addition, the proofs are very functional-analytic, as in~\cite{Bal08}, and (at least as claimed by Ozawa) they do not readily generalize to other dimensions. Ozawa himself calls for a probabilistic version of~his~result.

We believe that our approach {to} eigenvalue
fluctuations is exactly the kind called for by Ozawa. In particular, we
expect that several key steps underlying our proof of
Theorem~\ref{thm1.2} extend to the crushed-ice problem in all
dimensions. Notwithstanding, as the situation of independent and bounded potentials on a lattice is considerably simpler, we decided to start with that case first. Moreover, lattice Anderson Hamiltonians are well studied objects and so results for them are of interest in their own right. (See Subsection~\ref{sec2.3} for some more comments.)

\subsection{Random elliptic operators}\label{RCM}
In homogenization theory, the leading order of the eigenvalues of various random elliptic operators, whether in divergence form or not, has been studied quite thoroughly; see again the book by Jikov, Kozlov and Oleinik~\cite{ZKO}. An example of such operator (in divergence form) is the (scaled) random Laplacian
\begin{equation}
\cmss L^{(\epsilon)} f(x):=\frac12\,\epsilon^{-2}\!\sum_{y\colon|x-y|=1}c_{xy}\bigl[f(y)-f(x)\bigr]
\end{equation}
where $\{c_{xy}\colon x,y\in\Z^d,  |x-y|=1\}$ is a family of non-negative conductances with $c_{xy}=c_{yx}$.

We can naturally study the same question for the operator~$-\cmss L^{(\epsilon)}$ as we did for the Anderson Hamiltonian~\eqref{E:1.2}. Indeed, let~$\lambda_{D_\epsilon}^{\ssup k}$ denote the $k$-th eigenvalue of $-\cmss L^{(\epsilon)}$ on the linear space of functions that vanish outside the set~$D_\epsilon$ defined in \eqref{E:1.1}. Under the assumption that $(c_{xy})_{x\sim y}$ is ergodic with respect to spatial shift and uniformly elliptic in the sense that
\begin{equation}
\exists\, a,b\in(0,\infty),\,\,a<b\colon\qquad c_{xy}\in[a,b]{\qquad \text{almost surely,}}
\end{equation}
the eigenvalue $\lambda_{D_\epsilon}^{\ssup k}$ converges (in probability) to the $k$-th smallest eigenvalue of~$-\cmss Q$ on~$D$, where~$\cmss Q$ is the elliptic second-order differential operator
\begin{equation}
\cmss Q f(x):=\sum_{i,j=1}^d q_{ij}\frac{\partial^2 f}{\partial x_i\partial x_j}(x)
\end{equation}
with Dirichlet boundary conditions on~$\partial D$ and $(q_{ij})_{i,j}$ denoting a positive-definite  symmetric (constant) matrix. 

To the best of our knowledge, the fluctuations of $\lambda_{D_\epsilon}^{\ssup k}$ for independent and identically distributed conductances have not been studied yet. Notwithstanding, the analysis of a related effective conductance problem (Nolen~\cite{Nolen}, Rossignol~\cite{Rossignol}, Biskup, Salvi and Wolff~\cite{BSW}) indicates that $\epsilon^{-d/2}[\lambda_{D_\epsilon}^{\ssup k}-\E\lambda_{D_\epsilon}^{\ssup k}]$ should be asymptotically normal with mean zero and variance that is a biquadratic expression in~$\nabla\varphi^{\ssup k}_D$ integrated over~$D$, where $\varphi^{\ssup k}_D$ denotes a~$k$-th eigenfunction of the operator~$\cmss Q$. A significant additional technical challenge of this problem is the need to employ the corrector method (this is what gives rise to the ``homogenized'' coefficients $q_{ij}$ above).

\subsection{Anderson localization} 
\label{sec2.3}\noindent
Our discussion of the background would not be complete without making at least some connection to the problem of Anderson localization. The name goes back to the seminal (physics) 1958 article by Anderson~\cite{Anderson} who noted that metals may turn from conductors to insulators when impurities are inserted to the crystalline structure at sufficient density. Mathematically, the insulator phase refers to the situation when the infinite-volume version of the operator \eqref{E:1.2} with $\epsilon:=1$  exhibits a band of localized eigenvalues. (This is what is referred to as Anderson localization.) The conductor phase indicates the existence of a band of continuous spectrum. 

Through tremendous effort by mathematicians over the last four decades, Anderson localization has now been at least partially understood. Instead of trying to summarize the vast literature, we refer the reader to the monographs of Pastur and Figotin~\cite{Pastur-Figotin}, Stollmann~\cite{Stollmann}, Carmona and Lacroix~\cite{Carmona-Lacroix} and the notes by Hundertmark~\cite{Hundertmark}. The upshot is that one-dimensional models exhibit only localized states while all models exhibit localized states near ``spectral edge.'' The delocalized phase remains a complete mystery, being so far successfully tackled only in the case of tree graph models (cf.~the upcoming book by Aizenman and Warzel~\cite{Aizenman-Warzel}).

Another way to look at Anderson localization is by analyzing the limiting spectral statistics for operators in an increasing sequence of finite volumes. In the localized regime, the statistics is expected to be given by a Poisson point process. This has so far been proved in the ``bulk'' (i.e., the interior) of the spectrum {(Molchanov~
\cite{Molchanov} in~$d=1$ and Minami~\cite{Minami} for general $d\ge1$)}. At spectral edges there seem to be only partial results for bounded potentials at this time (Germinet and Klopp~\cite{GK1,GK2}) although a somewhat more complete theory has been developed for some unbounded potentials (Astrauskas~\cite{A1,A2}, Biskup and K\"onig~\cite{BK}).
In the delocalization regime, the spectral statistics is expected to be that seen in random matrix ensembles.

Having noted all these facts, we rush to add that the main point of our article is to describe the situation of a very weak disorder, which one can see by multiplying $H_{D_\epsilon,\xi}$ by~$\epsilon^2$. The effective strength of the random potential, and consequently also the effect of Anderson localization, vanishes in the limit~$\epsilon\downarrow0$. Notwithstanding, as for the crushed-ice problem, a residual term coming from smooth spatial variations of the mean (expressed by the function~$U$) prevails and the eigenvalues are asymptotically those of a non-trivial continuum Schr\"odinger operator.

\nopagebreak
\section{Convergence to continuum model}
\nopagebreak\label{sec3}\noindent
We are now in a position to start the expositions of the proofs. Our first task will be to prove Theorem~\ref{thm1.1} dealing with the leading-order convergence of the random eigenvalues to those of the continuum problem.  Let us begin by fixing some notation.

\subsection{Notations}
We will henceforth assume that~$D$ is a bounded open set in $\R^d$ with $C^{1,\alpha}$-boundary for some $\alpha>0$ and that Assumption~\ref{ass1} holds. We write
\begin{equation}
\Omega_{a,b}:=[a,b]^{\Z^d},
\end{equation}
{for a set that} supports~$\BbbP_\epsilon$ for every~$\epsilon>0$. Recalling the notation $g_{D_\epsilon,\xi}^{\ssup k}$ for the~$k$-th eigenvector of $H_{D_\epsilon,\xi}$ normalized as in~\eqref{E:3.1}, we similarly write $\varphi_D^{\ssup k}$ for an eigenfunction of $-\Delta+ U(x)$ corresponding to~$\lambda^{\ssup k}_D$ normalized so that $\int_D|\varphi_D^{\ssup k}(x)|^2\textd x=1$. These eigenfunctions are unique up to a sign as soon as the corresponding eigenvalue is non-degenerate.

We will write $\Vert f\Vert_{p}$ for the canonical $\ell^p$-norm of~$\R$- or~$\R^d$-valued functions $f$ on~$\Z^d$. When~$p=2$, we use $\langle f,h\rangle$ to denote the associated inner product in $\ell^2(\Z^d)$. All functions defined \emph{a priori} only on~$D_\epsilon$ will be regarded as extended by zero to~$\Z^d\smallsetminus D_\epsilon$.
In order to control convergence to the continuum problem, it will sometimes be convenient to work with the scaled $\ell^p$-norm,
\begin{equation}
\Vert f\Vert_{\epsilon,p}:=\biggl(\epsilon^d\sum_{x\in \Z^d}|f(x)|^p\biggr)^{\ffrac1p}.
\end{equation}
{This} implies, e.g., that
\begin{equation}
\Vert \epsilon^{-d/2}g_{D_{\epsilon},\xi}^{\ssup k}\Vert_{\epsilon,2}=1.
\end{equation}
We will sometimes use $\langle f,g\rangle_{\epsilon,2}$ to denote the inner product associated with~$\Vert\cdot\Vert_{\epsilon,2}$. For functions~$f,g$ of a continuum variable, we write the norms as~$\Vert f\Vert_{L^p(\R^d)}$ and the inner product in~$L^2(\R^d)$ as $\langle f,g\rangle_{L^2(\R^d)}$.

\subsection{Regularity bounds}
Our starting point are some regularity estimates on both the continuum and discrete eigenvalues and eigenfunctions. Note that, in our earlier convention,~$\lambda_{D_\epsilon,0}^{\ssup k}$ corresponds to the $k$-th eigenvalue of $-\epsilon^{-2}\Delta^{(\textd)}$ with Dirichlet boundary conditions on~$D_\epsilon^\cc$. Recall that $C^{1,\alpha}(A)$ denotes the set of functions that are continuously differentiable on the interior of~$A$ with a uniform estimate on~$\alpha$-H\"older norm of the gradient.

\begin{lemma}
\label{lemma-regularity}
For all $k\ge1$
\begin{equation}
\label{E:3.4aa}
\sup_{0<\epsilon<1}\,\sup_{\xi\in\Omega_{a,b}}\,\bigl|\lambda_{D_\epsilon,\xi}^{\ssup k}-\lambda_{D_\epsilon,0}^{\ssup k}\bigr|\le \max\{|b|,|a|\}.
\end{equation}
Similarly, both $-\Delta$ and $-\Delta+U(x)$ have compact resolvent on~$\cmss H^1_0(D)$ and their spectrum thus consists of isolated, finitely degenerate eigenvalues. Moreover, if $\lambda_{D,0}^{\ssup k}$ denotes the $k$-th eigenvalue of $-\Delta$ on~$\cmss H^1_0(D)$, then
\begin{equation}
\label{E:3.5aa}
\bigl|\lambda_{D}^{\ssup k}-\lambda_{{D},0}^{\ssup k}\bigr|\le \Vert U\Vert_\infty.
\end{equation}
In addition, any eigenfunction $\varphi^{\ssup k}_D$ of $-\Delta+U(x)$ obeys
\begin{equation}
\label{E:3.6aa}
\varphi^{\ssup k}_D\in C^{1,\alpha}(\overline D).
\end{equation}
\end{lemma}

\begin{proofsect}{Proof}
The estimates \twoeqref{E:3.4aa}{E:3.5aa} are consequences of the Minimax Theorem. The regularity of the eigenfunction follows from the regularity of the boundary of $D$ via, e.g., Corollary 8.36 of Gilbarg and Trudinger~\cite{GT}.
\end{proofsect}

The following estimate will be quite convenient for the derivations in the rest of the paper:

\begin{lemma}
\label{prop4}
For~$k\ge1$, there is a constant~$c=c(k,a,b,D)$, such that
\begin{equation}
\sup_{\xi\in\Omega_{a,b}}\,\Vert g_{D_\epsilon,\xi}^{\ssup k}\Vert_{\infty}\le c\,\epsilon^{d/2}.
\end{equation}
\end{lemma}

\begin{proofsect}{Proof}
Let~$g$ be an eigenfunction of~$H_{D_\epsilon,\xi}$ for an eigenvalue~$\lambda$ normalized so that $\Vert g\Vert_{2}=1$. The key observation is that the inner product $\langle\delta_x,\texte^{t\Delta^{(\textd)}}\delta_y\rangle$, with $\Delta^{(\textd)}$ taken with respect to the Dirichlet boundary condition, coincides with the transition probability $p_t(x,y)$ of a continuous-time (constant-speed) simple random walk on~$\Z^d$ killed upon exit from~$D_\epsilon$. The eigenvalue equation and the Feynman-Kac formula imply
\begin{equation}
\begin{aligned}
g(x)&=\texte^{\lambda t}\bigl(\texte^{t\epsilon^{-2}(\Delta^{(\textd)}-\epsilon^2\xi)}g\bigr)(x)
\\
&=\texte^{\lambda t} E^x\biggl(\exp\Bigl\{\int_0^{\epsilon^{-2}t}\epsilon^2\xi(X_s)\textd s\Bigr\}g(X_{t\epsilon^{-2}})\biggr),
\end{aligned}
\end{equation}
where the expectation is over random walks $(X_s)$ started at~$x$. Taking absolute values, bounding $|\xi(x_i)|$ by $\Vert\xi\Vert_\infty$ and writing the result using the semigroup, we get
\begin{equation}
\bigl|g(x)\bigr|\le\texte^{(\lambda+\Vert\xi\Vert_\infty) t}\sum_{y\in D_\epsilon}p_{\epsilon^{-2}t}(x,y)\bigl|g(y)\bigr|.
\end{equation}
Applying the Cauchy-Schwarz inequality and using that~$g$ is normalized yields
\begin{equation}
\label{E:3.10aa}
g(x)^2\le\texte^{2(\lambda+\Vert\xi\Vert_\infty) t}\sum_{y\in D_\epsilon}p_{\epsilon^{-2}t}(x,y)^2\le
\texte^{2(\lambda+\Vert\xi\Vert_\infty) t}p_{2\epsilon^{-2}t}(x,x),
\end{equation}
where the second inequality follows by the fact that $p_t$ is reversible with respect to the counting measure. But $p_{t}(x,x)$ is non-decreasing in~$D_\epsilon$ and so it is bounded by the corresponding quantity on~$\Z^d$. The local central limit theorem (or other methods to control heat kernels) then yield $p_t(x,x)\le Ct^{-d/2}$ for all~$t\ge1$. Setting $t:=1$ in \eqref{E:3.10aa}, the claim follows.
\end{proofsect}

Note that Lemma~\ref{prop4} and the fact that $|D_\epsilon|=O(\epsilon^{-d})$ imply
\begin{equation}
\sup_{p\in[1,\infty]}\,\sup_{0<\epsilon<1}\,\sup_{\xi\in\Omega_{a,b}}\,\Vert \epsilon^{-d/2} g_{D_\epsilon,\xi}^{\ssup k}\Vert_{\epsilon,p}<\infty
\end{equation}
for all $k\ge1$.

\subsection{Continuum interpolation}
Having dispensed with regularity issues, we now proceed to develop tools that will help us approximate discrete eigenfunctions by continuous ones. The piece-wise constant approximation is a natural first candidate: For any function $f\colon\Z^d\to\R$, set
\begin{equation}
\bar f(x):=\epsilon^{-d/2}f\bigl(\lfloor x/\epsilon\rfloor\bigr),\qquad x\in\R^d.
\end{equation}
The scaling ensures that, automatically,~$\langle f,h\rangle=\langle\bar f,\overline h\rangle_{L^2(\R^d)}$. Unfortunately, our need to control the kinetic energy makes this approximation less attractive in detailed estimates. Instead, we will use an approximation by piece-wise linear interpolations over lattice cells. The following lemma can be extracted from the proof of Lemma~2.1 in Becker and K\"onig~\cite{Becker-Koenig}:

\begin{lemma}
\label{fe}
There is a constant~$C=C(d)$ for which the following holds: For
any function $f\colon\Z^d\to\R$ and any~$\epsilon\in(0,1)$, there is a function $\wt f\colon \R^d\to\R$ such that\begin{enumerate}
\item the map $f\mapsto\wt f$ is linear,
\item $\wt f$ is continuous on~$\R^d$ and $\wt f(x\epsilon)=f(x)$ for all $x\in \Z^d$,
\item
for any $x\in\Z^d$ and any $y\in \epsilon x+[0,\epsilon)^d$ we have
\begin{equation}
\label{E:3.13aa}
\bigl|\wt f(y)\bigr|\le\max_{z\in x+\{0,1\}^d}\,\bigl|f(z)\bigr|,
\end{equation}
and
\begin{equation}
\label{E:3.14uu}
\bigl|\wt f(y)-f(x)\bigr|\le d\max_{z\in x+\{0,1\}^d}\bigl|\nabla^{(\textd)}f(z)\bigr|,
\end{equation}
\item
for all~$p\in[1,\infty]$ we have
\begin{equation}
\label{E:3.15uu}
\Vert\wt f\Vert_{L^p(\R^d)}\le C(d)\Vert f\Vert_{\epsilon,p},
\end{equation}
and
\begin{equation}
\label{E:3.16q}
\Bigl|\,\Vert \wt f\Vert_{L^2(\R^d)}-\Vert f\Vert_{\epsilon,2}\Bigr|\le C(d)\|\nabla^{(\textd)} 
f\|_{\epsilon,2},
\end{equation}
\item
$\wt f$ is piece-wise linear and thus a.e.\ differentiable with
\begin{equation}
\label{E:3.15q}
\|\nabla\wt f\|_{L^2(\R^d)}
=\epsilon^{-1}\|\nabla^{(\textd)} 
f\|_{\epsilon,2}.
\end{equation}
\end{enumerate}
\end{lemma}
\begin{proof}
Although most of these are already contained in the proof of \cite[Lemma~2.1]{Becker-Koenig}, we provide an independent proof as the desired statements are hard to glean from the notations used there. A key point is that for any $y=(y_1,\dots,y_d)\in[0,1)^d$ there is a permutation~$\sigma$ of~$\{1,\dots,d\}$ such that $y_{\sigma(1)}\ge \dots\ge y_{\sigma(d)}$. Moreover, when all components of~$y$ are distinct, such a~$\sigma$ is unique. 

Given $y\in x\epsilon+[0,\epsilon)^d$ let thus $\sigma$ be a permutation
 that puts the components of $y-x\epsilon$ in non-increasing
 ordering. Writing the reordered components of $y/\varepsilon-x$ as $1\ge\alpha_1\ge\dots\ge\alpha_d\ge0$, we have
\begin{equation}
\label{E:3.15aa}
y=x\epsilon+\epsilon\sum_{i=1}^d\alpha_i\,\hate_{\sigma(i)}.
\end{equation}
We then define
\begin{equation}
\label{E:3.16aa}
\wt f(y):=f(x)+\sum_{i=1}^d\alpha_i\bigl(\nabla^{(\textd)}_{\sigma(i)}f\bigr)(x+\hate_{\sigma(1)}+\dots+\hate_{\sigma(i-1)}),
\end{equation}
where, we recall, $(\nabla^{{(\textd)}}_if)(x):=f(x+\hate_i)-f(x)$. 

Our first task is to check that~$\wt f$ is well defined. Obviously, the $\alpha_j$'s are determined by~$y$ so we only have to check that the definition does not depend on~$\sigma$, if there is more than one for the same~$y$. That happens only when~$\alpha_i=\alpha_{i+1}$ for some $i=0,\dots,d-1$ (where $\alpha_0:=1$ by convention). Then \eqref{E:3.15aa} holds also for~$\sigma$ replaced by permutation $\sigma'$ which agrees with~$\sigma$ except at indices~$i,i+1$ where $\sigma'(i):=\sigma(i+1)$ and $\sigma'(i+1):=\sigma(i)$. Abbreviating~$z:=x+\hate_{\sigma(1)}+\dots+\hate_{\sigma(i-1)}$, the two possible expressions for~$\wt f(y)$ will agree if and only if
\begin{equation}
(\nabla_{\sigma(i)}f)(z)+(\nabla_{\sigma(i+1)}f)(z+\hate_{\sigma(i)})
=(\nabla_{\sigma(i+1)}f)(z)+(\nabla_{\sigma(i)}f)(z+\hate_{\sigma(i+1)}).
\end{equation}
As is readily verified, both of these are equal to
 $f(z+\hate_{\sigma(i)}+\hate_{\sigma(i+1)})-f(z)$. Hence,~$\wt f$
is consistent. The map $f\mapsto\wt f$ is obviously linear, thus proving~(1).

We now move to checking continuity of~$\wt f$. First note that \eqref{E:3.16aa} extends to all points in the closed ``cube'' $\CC(x):=x\epsilon+[0,\epsilon]^d$. In light of uniform continuity of~$\wt f$ on the open ``cube,'' the extension is continuous, and thus independent of~$\sigma$ (if more than one~$\sigma$ corresponds to the same boundary point). Now pick~$y\in\CC(x)\cap\CC(x+\hate_i)$. As $f(x)+\nabla^{(\textd)}_i f(x)=f(x+\hate_i)$, taking \eqref{E:3.16aa} on~$\CC(x)$ with $\sigma(1):=i$ and $\alpha_1:=1$  has the same value as \eqref{E:3.16aa} on $\CC(x+\hate_i)$ with $\sigma(d):=i$ and~$\alpha_d:=0$. Hence, the expressions for~$\wt f$ on $\CC(x)$ and~$\CC(x+\hate_i)$ agree on on the common ``side'' $\CC(x)\cap\CC(x+\hate_i)$ and~$\wt f$ is thus continuous on~$\R^d$. Conclusion~(2) is readily checked.

It remains to prove the stated bounds. For that we first note that \eqref{E:3.16aa} can be recast as
\begin{equation}
\wt f(y)=\sum_{i=0}^d(\alpha_i-\alpha_{i+1})f\bigl(x+\hate_{\sigma(1)}+\dots+\hate_{\sigma(i)}\bigr)
\end{equation}
where~$\alpha_0:=1$ and $\alpha_{d+1}:=0$. Using that $\alpha_i-\alpha_{i+1}$ are non-negative and sum up to one, we get \eqref{E:3.13aa}. This immediately yields \eqref{E:3.15uu}. Similarly, \eqref{E:3.16aa} and the fact that $|\alpha_i|\le1$ directly show~\eqref{E:3.14uu}. To get \eqref{E:3.16q} from this, abbreviate $h(y):=\wt f(y)-f(\lfloor y/\epsilon\rfloor)$. Squaring \eqref{E:3.14uu}, bounding the maximum (of squares) by a sum and integrating over $y\in\R^d$ yields
\begin{equation}
\Vert h\Vert_{L^2(\R^d)}\le C(d)\Vert \nabla^{(\textd)}f\Vert_{\epsilon,2}.
\end{equation}
But the $L^2$-norm of~$y\mapsto f(\lfloor y/\epsilon\rfloor)$ is $\Vert f\Vert_{\epsilon,2}$ and so we get \eqref{E:3.16q} by the triangle inequality.

Concerning \eqref{E:3.15q}, define $W_\sigma:=\bigcup_{x\in\Z^d}\{\epsilon x+z\colon z\in[0,\epsilon)^d,\,z_{\sigma(1)}>\dots>z_{\sigma(d)}\}$ and note that~$\wt f$ is piece-wise linear on~$W_\sigma$ with
\begin{equation}
\nabla_{\sigma(i)} \wt f(y)=\epsilon^{-1}(\nabla_{\sigma(i)}^{(\textd)}f)\bigl(\lfloor y/\epsilon\rfloor+\hate_{\sigma(1)}+\dots+\hate_{\sigma(i-1)}\bigr),\qquad y\in W_\sigma.
\end{equation}
This implies
\begin{equation}
\int_{W_\sigma}\bigl|\nabla f(y)\bigr|^2\textd y = \epsilon^{-2}\sum_{i=1}^d\sum_{x\in\Z^d}\bigl|(\nabla_{\sigma(i)}^{(\textd)}f)(x)\bigr|^2\int\1_{\{1\ge\alpha_1>\dots>\alpha_d\ge0\}}\textd\alpha_1\dots\textd\alpha_d.
\end{equation}
The integral on the right equals $(d!)^{-1}$ so we get \eqref{E:3.15q} by summing over all admissible~$\sigma$ and using that $W_\sigma$'s cover~$\R^d$ up to a set of zero Lebesgue measure.
\end{proof}

Our next item of concern is an approximation of functions on the lattice by piecewise constant modifications. For each~$L\ge1$ and any $f\colon\Z^d\to\R$, {denote}
\begin{equation}
f_L(x):=f\bigl(L\lfloor x/L\rfloor\bigr).
\end{equation}
{Then we have:}

\begin{lemma}
\label{lemma-3.4new}
There exists a  constant $C(d)<\infty$  such that{, for any $L\ge1$ and any $f\colon\Z^d\to\R$,}
\begin{equation}
\Vert f-f_L\Vert_1<C(d)L\,\Vert\nabla^{(\textd)}f\Vert_1.
\end{equation}
\end{lemma}

\begin{proofsect}{Proof}
Consider the box $B_k:=x_0+\{0,\dots,k-1\}^d$. The triangle inequality shows
\begin{equation}
\sum_{x\in B_k\smallsetminus B_{k-1}}\bigl|f(x)-f(x_0)\bigr|
\le
\sum_{x\in B_{k-1}\smallsetminus B_{k-2}}\biggl(\,\bigl|f(x)-f(x_0)\bigr|+\sum_{z\in\{0,1\}^d}\sum_{i=1}^d\bigl|(\nabla^{(\textd)}_if)(x+z)\bigr|\biggr).
\end{equation}
This implies 
\begin{equation}
\sum_{x\in B_L}\bigl|f(x)-f(x_0)\bigr|\le
2^d\sqrt d\,L\sum_{x\in B_L}\bigl|(\nabla^{(\textd)}f)(x)\bigr|.
\end{equation}
The claim follows by summing over $x_0\in(L\Z)^d$.
\end{proofsect}

\subsection{Convergence of eigenfunctions/eigenvalues}
We will now proceed to tackle convergence statements. We will employ a standard trick: Instead of individual eigenvalues, we will work with their sums
\begin{equation}
\Lambda_k^{\epsilon}(\xi):=\sum_{i=1}^k\lambda^{\ssup i}_{D_\epsilon,\xi}
\quad\text{and}\quad
\Lambda_k:=\sum_{i=1}^k\lambda^{\ssup i}_{D}.
\end{equation}
These quantities are better suited for dealing with degeneracy because they are concave in~$\xi$ and, in fact,  admit a variational characterization (sometimes dubbed the Ky Fan Maximum Principle~\cite{KyFan}) of the form
\begin{equation}
\label{E:3.22}
\Lambda_k^{\epsilon}(\xi)=\,\inf_{\begin{subarray}{c}
h_1,\dots,h_k\\\textrm{ONS}
\end{subarray}}\,\,
\sum_{i=1}^k\bigl(\epsilon^{-2}\|\nabla^{(\textd)}h_i\|_2^2+\langle\xi, h_i^2\rangle\bigr)
\end{equation}
and
\begin{equation}
\label{E:3.22b}
\Lambda_k=\,\inf_{\begin{subarray}{c}
\psi_1,\dots,\psi_k\\\textrm{ONS}
\end{subarray}}\,\,
\sum_{i=1}^k\bigl(\,\|\nabla\psi_i\|_{L^2(\R^d)}^2+\langle U, \psi_i^2\rangle_{L^2(\R^d)}\bigr).
\end{equation}
Here the acronym ``ONS'' indicates that the $k$-tuple of functions form
an orthonormal system in the subspace corresponding to Dirichlet boundary
conditions (and, in the latter case, also tacitly assumes that the functions are in
the domain of the gradient). Substituting actual eigenfunctions shows
that the sums of eigenvalues are no smaller than the infima but the
complementary bound requires a bit of work. The argument actually yields a quantitative form of the Ky Fan Maximum Principle which will be quite suitable for our later needs:

\begin{lemma}
\label{lemma-projections}
Consider a separable Hilbert space~$\HH$ and a self-adjoint linear operator~$\hat H$ on~$\HH$ which is bounded from below and has compact resolvent. Let~$\{\varphi_i\colon i\ge1\}$ be an orthonormal basis of eigenfunctions of~$\hat H$ corresponding to eigenvalues $\lambda_i$ that we assume obey $\lambda_{i+1}\ge\lambda_i$ for all~$i\ge1$. Let~$\hat \Pi_k$ denote the orthogonal projection onto $\{\varphi_1,\dots,\varphi_k\}^{\perp}$. Then for any ONS $\psi_1,\dots,\psi_k$ that lies in the domain of~$\hat H$,
\begin{equation}
\label{E:3.28aa}
\sum_{i=1}^k\langle \psi_i,\hat H\psi_i\rangle-(\lambda_1+\dots+\lambda_k)\ge
(\lambda_{k+1}-\lambda_k)\sum_{i=1}^k\Vert \hat\Pi_k\psi_i\Vert^2.
\end{equation}
\end{lemma}

\begin{proofsect}{Proof}
We provide a proof as it is very short. The argument parallels the derivation of Lemma~3.2 in Barekat~\cite{Barekat}.
Since $\psi_1,\dots,\psi_k$ is an ONS and~$\HH$ is separable, we may extend it into an orthonormal (countable) basis $\{\psi_i\colon i\ge1\}$. Denoting $a_{ij}:=\langle\psi_i,\varphi_j\rangle$, the Parseval identity yields
\begin{equation}
b_j:=\sum_{i=1}^k|a_{ij}|^2\le\sum_{i\ge1}|a_{ij}|^2=\langle\varphi_j,\varphi_j\rangle=1.
\end{equation}
Since $\sum_{j\ge1}b_j=k$, we have $\sum_{j>k}b_j=\sum_{j=1}^k(1-b_j)$ and it thus follows that
\begin{equation}
\begin{aligned}
\sum_{i=1}^k\langle \psi_i,\hat H\psi_i\rangle
&=\sum_{i=1}^k\sum_{j\ge1}\lambda_j|a_{ij}|^2=
\sum_{j\ge1}b_j\lambda_j
\\
&\ge\sum_{j=1}^k\lambda_j b_j+\lambda_{k+1}\sum_{j>k}b_j
\\
&=\lambda_1+\dots+\lambda_k+\sum_{j=1}^k(\lambda_{k+1}-\lambda_j)(1-b_j)
\\
&\ge
\lambda_1+\dots+\lambda_k+(\lambda_{k+1}-\lambda_k)\sum_{j=1}^k(1-b_j).
\end{aligned}
\end{equation}
Writing the last sum as $\sum_{j>k}b_j$ we easily see that it equals $\sum_{i=1}^k\Vert\hat\Pi_k\psi_i\Vert^2$.
\end{proofsect}

Our next goal, formulated in Propositions~\ref{lemma-3.4} and~\ref{lemma3.5} below, is to establish convergence $\Lambda_k^{\epsilon}(\xi)\to\Lambda_k$ in probability. Throughout we assume the setting in Assumption~\ref{ass1}.

\begin{proposition}
\label{lemma-3.4}
For any~$\delta>0$,
\begin{equation}
\lim_{\epsilon\downarrow0}\,\BbbP_\epsilon\bigl(\Lambda_k^{\epsilon}(\xi)\ge\Lambda_k+\delta\bigr)=0.
\end{equation}
\end{proposition}

\begin{proofsect}{Proof}
Consider (a choice of) an {ONS} of the first~$k$ eigenfunctions $\varphi_D^{\ssup 1},\dots,\varphi_D^{\ssup k}$ of~$-\Delta+U$. By Lemma~\ref{lemma-regularity} all of these are $C^{1,\alpha}$. Now define
\begin{equation}
\label{E:3.32}
f_i(x):=
\begin{cases}
\varphi_D^{\ssup i}(x\epsilon),\qquad&\text{if }x\in D_\epsilon,
\\
0,\qquad&\text{otherwise}.
\end{cases}
\end{equation}
{Thanks} to uniform continuity of the eigenfunctions, we then have
\begin{equation}
\langle f_i,f_j\rangle_{\epsilon,2}\,\underset{\epsilon\downarrow0}\longrightarrow\,\langle \varphi_D^{\ssup i},\varphi_D^{\ssup j}\rangle_{L^2(D)}=\delta_{ij}
\end{equation}
and so for~$\epsilon$ small the functions {$f_1,\dots,f_k$ are nearly
 mutually } orthogonal. Applying the Gram-Schmidt orthogonalization
 procedure, {we see that} there are  functions
 {$\{h_i^{\epsilon}\}_{i=1}^k$} and coefficients
 {$\{a_{ij}(\epsilon)\}_{1\le i,j\le k}$} such that
\begin{equation}
\label{E:3.35}
h_i^\epsilon=\sum_{j=1}^k(\delta_{ij}+a_{ij}(\epsilon))f_j,\qquad i=1,\dots,k,
\end{equation}
with
\begin{equation}
\label{E:3.36}
\langle h_i^\epsilon,h_j^\epsilon\rangle_{\epsilon,2}=\delta_{ij}
\quad\text{and}\quad\max_{i,j}|a_{ij}(\epsilon)|\,\underset{\epsilon\downarrow0}\longrightarrow\,0.
\end{equation}
Moreover, the definition of~$f_i$ and the $C^{1,\alpha}$-regularity of the eigenfunctions imply
\begin{equation}
\label{E:3.33}
\sup_{\begin{subarray}{c}
y\in D\\\dist_\infty(y,D^\cc)>2\epsilon
\end{subarray}}
\Bigl|\nabla\varphi_D^{\ssup i}(y)-\epsilon^{-1}(\nabla^{(\textd)}f_i)(\lfloor y/\epsilon\rfloor)\Bigr|\,\underset{\epsilon\downarrow0}\longrightarrow\,0
\end{equation}
and same continues to hold for $h_i^\epsilon$ {instead of $f_i$} as well. 
Hereby we get
\begin{equation}
\epsilon^{-1}\Vert\nabla^{(\textd)}h_i^\epsilon\Vert_{\epsilon,2}\,\underset{\epsilon\downarrow0}\longrightarrow\,
\Vert\nabla\varphi_D^{\ssup i}\Vert_{L^2(\R^d)}
\end{equation}
and, by continuity of~$U$, also
\begin{equation}
\bigl\langle U(\epsilon\cdot),(h_i^\epsilon)^2\bigr\rangle_{\epsilon,2}
\,\underset{\epsilon\downarrow0}\longrightarrow\,
\langle U,\varphi_D^{\ssup i}\rangle_{L^2(\R^d)}.
\end{equation}
Once the two sides in each of these limit statements (for all~$i=1,\dots,k$) are within some~$\delta\in(0,1)$ of each other, the variational characterization \eqref{E:3.22} yields
\begin{equation}
\Lambda_k^{\epsilon}(\xi)\le\Lambda_k+2k\delta+\sum_{i=1}^k\bigl\langle \xi-U(\epsilon\cdot),(h_i^\epsilon)^2\bigr\rangle_{\epsilon,2}.
\end{equation}
Invoking a union bound we obtain
\begin{equation}
\BbbP_\epsilon\bigl(\Lambda_k^{\epsilon}(\xi)\ge\Lambda_k+3k\delta\bigr)
\le\sum_{i=1}^k\BbbP_\epsilon\Bigl(\,\bigl\langle \xi-U(\epsilon\cdot),(h_i^\epsilon)^2\bigr\rangle_{\epsilon,2}\ge\delta\Bigr).
\end{equation}
The Chebyshev inequality now shows
\begin{equation}
\BbbP_\epsilon\Bigl(\,\bigl\langle \xi-U(\epsilon\cdot),(h_i^\epsilon)^2\bigr\rangle_{\epsilon,2}\ge\delta \Bigr)
\le\frac{C}{\delta^2}\sum_{x\in D_\epsilon}\epsilon^{2d}h_i^\epsilon(x)^4,
\end{equation}
where~$C$ is a uniform bound on $\text{Var}(\xi(x))$. But the $h_i^\epsilon$'s are bounded and since $\Vert h_i^\epsilon\Vert_{\epsilon,2}=1$, the right-hand side is proportional to~$\epsilon^d$. As~$\delta$ was arbitrary, the claim follows.
\end{proofsect}

\begin{proposition}
\label{lemma3.5}
For any~$\delta>0$,
\begin{equation}
\lim_{\epsilon\downarrow0}\,\BbbP_\epsilon\bigl(\Lambda_k^{\epsilon}(\xi)\le\Lambda_k-\delta\bigr)=0.
\end{equation}
\end{proposition}

\begin{proofsect}{Proof}
Let $g_{D_\epsilon,\xi}^{\ssup 1},\dots,g_{D_\epsilon,\xi}^{\ssup k}$ be
 (a choice of) {an ONS of} the first $k$ eigenfunctions of~$H_{D_\epsilon,\xi}$ 
and let $\wt g_{1,\xi}^{\,\epsilon},\dots,\wt g_{k,\xi}^{\,\epsilon}$ denote the continuum interpolations of $\epsilon^{-d/2}g_{D_\epsilon,\xi}^{\ssup 1},\dots,\epsilon^{-d/2}g_{D_\epsilon,\xi}^{\ssup k}$, respectively, as {described in} Lemma~\ref{fe}. The uniform bound \eqref{E:3.4aa} on the eigenvalues ensures
\begin{equation}
\label{E:3.47}
\sup_{\xi\in \Omega_{a,b}}\,\,\sup_{0<\epsilon<1} \epsilon^{-1}\Vert \nabla^{(\textd)}g_{D_\epsilon,\xi}^{\ssup i}\Vert_{2}<\infty
\end{equation}
and so, in light of Lemma~\ref{fe}(4),
\begin{equation}
\sup_{\xi\in\Omega_{a,b}}\Bigl|\langle\wt g_{i,\xi}^{\,\epsilon},\wt g_{j,\xi}^{\,\epsilon}\rangle_{L^2(\R^d)}-\delta_{ij}\Bigr|\,\underset{\epsilon\downarrow0}\longrightarrow\,0.
\end{equation}
Invoking again the Gram-Schmidt orthogonalization, we can thus find functions $\wt h_{1,\xi}^\epsilon,\dots,\wt h_{k,\xi}^\epsilon$ and coefficients $a_{ij}(\xi,\epsilon)$ such that
\begin{equation}
\label{E:3.45}
\wt h_{i,\xi}^\epsilon=\sum_{j=1}^k\bigl(\delta_{ij}+a_{ij}(\xi,\epsilon)\bigr)\wt g_{i,\xi}^{\,\epsilon},\qquad i=1,\dots,k,
\end{equation}
for which
\begin{equation}
\label{E:3.46}
\bigl\langle \wt h_{i,\xi}^\epsilon,{\wt h}_{j,\xi}^\epsilon\bigr\rangle_{L^2(\R^d)}=\delta_{ij}
\qquad\text{and}\qquad\max_{ij}\sup_{{\xi}\in\Omega_{a,b}}\,\bigl|a_{ij}(\xi,\epsilon)\bigr|\,\underset{\epsilon\downarrow0}\longrightarrow\,0.
\end{equation}
Thanks to the definition of~$D_\epsilon$, both the $\wt g_{i,\xi}^{\,\epsilon}$'s and $\wt h_{i,\xi}^{\epsilon}$'s are supported in~$D$.

Lemma~\ref{fe}(5), {\eqref{E:3.47}} and \twoeqref{E:3.45}{E:3.46} guarantee
\begin{equation}
\sup_{\xi\in\Omega_{a,b}}\,\Bigl|\,\Vert \nabla \wt h_{i,\xi}^{\epsilon}\Vert_{L^2(\R^d)}^2-\epsilon^{-2}\Vert \nabla^{(\textd)}g_{D_\epsilon,\xi}^{\ssup i}\Vert_{2}^2
\Bigr|\,\underset{\epsilon\downarrow0}\longrightarrow\,0
\end{equation}
while \eqref{E:3.14uu} ensures
\begin{equation}
\sup_{\xi\in\Omega_{a,b}}\,\Bigl|\bigl\langle U,(\wt h_{i,\xi}^{\epsilon})^2\bigr\rangle_{L^2({\R^d})}
-\bigl\langle U(\epsilon\cdot),(g_{D_\epsilon,\xi}^{\ssup i})^2\bigl\rangle\Bigr|\,\underset{\epsilon\downarrow0}\longrightarrow\,0.
\end{equation}
Once both suprema on the left are less than some $\delta>0$, using the $\wt h_{i,\xi}^{\epsilon}$ as the $\psi_i$'s in \eqref{E:3.22b} and noting that the $g_{D_\epsilon,\xi}^{\ssup i}$'s achieve the infimum in \eqref{E:3.22}, yields
\begin{equation}
\Lambda_k\le\Lambda_k^\epsilon(\xi)+2k\delta+\sum_{i=1}^k\bigl\langle U(\epsilon\cdot)-\xi,(g_{D_\epsilon,\xi}^{\ssup i})^2\bigr\rangle.
\end{equation}
Now consider the piece-wise constant approximation $f_L(x)=f(L\lfloor x/L\rfloor)$ to the function $f(x):=(g_{D_\epsilon,\xi}^{\ssup i}(x))^2$. Since $\Vert \nabla^{(\textd)} (g^2)\Vert_1\le C(d)\Vert g\Vert_2\Vert\nabla^{(\textd)} g\Vert_2$, Lemma~\ref{lemma-3.4new}, \eqref{E:3.47} and the boundedness of~$U-\xi$ give
\begin{equation}
\bigl\langle U(\epsilon\cdot)-\xi,(g_{D_\epsilon,\xi}^{\ssup i})^2\bigr\rangle
\le \bigl\langle U(\epsilon\cdot)-\xi,((g_{D_\epsilon,\xi}^{\ssup i})^2)_L\bigr\rangle
+CL\epsilon
\end{equation}
for some~$C$ independent of~$\xi$. Setting $B_L(x):=Lx+\{0,\dots,L-1\}^d$, on the event
\begin{equation}
\label{E:3.55}
F_{L,\epsilon}:=\bigcap_{\begin{subarray}{c}
x\in (L\Z)^d\\B_L(x)\cap D_\epsilon\ne\emptyset
\end{subarray}}
\biggl\{\xi\colon\Bigl|\,\sum_{z\in B_L(x)}U({z}\epsilon)-\xi({z})\Bigr|<\delta L^d\biggr\}
\end{equation}
we in turn have
\begin{equation}
\bigl\langle U(\epsilon\cdot)-\xi,((g_{D_\epsilon,\xi}^{\ssup i})^2)_L\bigr\rangle
\le\delta{(1+CL\epsilon)},
\end{equation}
again by Lemma~\ref{lemma-3.4new}. Assuming that $CL\epsilon\le\delta$, we thus get
\begin{equation}
\BbbP\bigl(\Lambda_k\ge\Lambda_k^\epsilon(\xi)+5k\delta\bigr)\le \BbbP_\epsilon(F_{L,\epsilon}^\cc).
\end{equation}
A standard large-deviation estimate bounds $\BbbP_\epsilon(F_{L,\epsilon}^\cc)\le c(\epsilon L)^{-d}\texte^{-cL^d}$. Choosing, e.g., $L=c\delta/\epsilon$ for some~$c$ sufficiently small, the claim follows.
\end{proofsect}

We are now ready to conclude:

\begin{proofsect}{Proof of Theorem~\ref{thm1.1}}
By Propositions~\ref{lemma-3.4} and~\ref{lemma3.5} we have\begin{equation}
\Lambda_k^\epsilon(\xi)\,\underset{\epsilon\downarrow0}{\overset{\BbbP}\longrightarrow}\,\Lambda_k,\qquad k\ge1.
\end{equation}
Then
\begin{equation}
\lambda_{D_\epsilon,\xi}^{\ssup k}=
\Lambda_k^\epsilon(\xi)-\Lambda_{k-1}^\epsilon(\xi)
\,\underset{\epsilon\downarrow0}{\overset{\BbbP}\longrightarrow}\,\Lambda_k-\Lambda_{k-1}=\lambda_D^{\ssup k}
\end{equation}
for all $k\ge1$ as well.
\end{proofsect}

The proof of Proposition~\ref{lemma3.5} gives us the following additional fact:

\begin{corollary}
\label{cor3.8}
Given any choice of $\xi\mapsto g_{D_\epsilon,\xi}^{\ssup 1}, \dots,
 g_{D_\epsilon,\xi}^{\ssup k}$, let $\wt
 g_{1,\xi}^{\,\epsilon},\dots,\wt g_{k,\xi}^{\,\epsilon}$ denote the
 continuum interpolations of $\epsilon^{-d/2}g_{D_\epsilon,\xi}^{\ssup
 1},\dots,\epsilon^{-d/2}g_{D_\epsilon,\xi}^{\ssup k}$ as constructed in
 Lemma~\ref{fe}. Assume $\lambda_D^{\ssup{k+1}}>\lambda_D^{\ssup k}$
 and let~$\hat\Pi_k$ denote the orthogonal projection on
 $\{\varphi_D^{\ssup1},\dots,\varphi_D^{\ssup
 k}\}^\perp$. 
Then, for any $\delta'>0$, there is an event~$E_{k,\epsilon,{\delta'}}$ such that
\begin{equation}
\label{E:3.60}
\biggl\{\,\xi\colon\sum_{i=1}^k\Vert\hat\Pi_k\wt g_{i,\xi}^{\,\epsilon}\Vert_{L^2({\R^d})}>{\delta'}\biggr\}\subseteq E_{k,\epsilon,{\delta'}}\qquad\text{and}\qquad\lim_{\epsilon\downarrow0}\,\BbbP_\epsilon(E_{k,\epsilon,{\delta'}})=0.
\end{equation}
\end{corollary}

\begin{proofsect}{Proof}
An inspection of the proof of Proposition~\ref{lemma3.5} reveals that
\begin{equation}
\left\{\xi:\sum_{i=1}^k\Bigl(\Vert \nabla \wt h_{i,\xi}^{\epsilon}\Vert_{L^2(\R^d)}^2+\bigl\langle U,(\wt h_{i,\xi}^{\epsilon})^2\bigr\rangle_{L^2({\R^d})}\Bigr)
\ge \Lambda_k-{\delta'}\right\}
\end{equation}
is a subset of the event $E_{k,\epsilon,{\delta'}}:=F_{L,\epsilon}^\cc$, where~$F_{L,\epsilon}$ is the event in  \eqref{E:3.55} with proper choices of~$\delta$ and~$L$. 
Thanks to Lemma~\ref{lemma-projections}, the inclusion in \eqref{E:3.60} thus holds for $\wt h_{i,\xi}^{\epsilon}$ instead of $\wt g_{i,\xi}^{\,\epsilon}$.
Adjusting~$\delta$ slightly, the identities \twoeqref{E:3.45}{E:3.46} then yield the same for the $\wt g_{i,\xi}^{\,\epsilon}$'s.
\end{proofsect}

\begin{remark}
Note that under the assumption $\lambda_D^{\ssup{k+1}}>\lambda_D^{\ssup k}$
the space $\{\varphi_D^{\ssup1},\dots,\varphi_D^{\ssup k}\}^\perp$, {and thus also the projection $\hat\Pi_k$}, is  independent of the choice of the {eigenfunction} basis. 
The formulation \eqref{E:3.60} avoids {having to deal with} questions about the {measurability} of eigenfunctions and/or the Hilbert-space projections.
\end{remark}

\section{Concentration estimate}
\label{sec4}\noindent
We now move to the proof of a concentration estimate for eigenfunctions around their mean. The proof actually boils down to a well-known concentration inequality due to Talagrand that we recast into a form adapted to our needs: 

\begin{theorem}[Theorem~6.6 of \cite{Tal96}]
\label{thm-Talagrand}
Let $N\in \N$ and let $|\cdot|_2$ denote the Euclidean norm on~$\R^N$. Let $f\colon [-1,1]^N\to \R$ be concave and Lipschitz continuous with
\begin{equation}
\label{E:3.19}
L:=\sup_{\xi,\eta\in[-1,1]^N}\frac{|f(\xi)-f(\eta)|}{|\xi-\eta|_2}<\infty.
\end{equation}
Then for any product probability measure $P$ on $[-1,1]^N$ and any~$t>0$,
\begin{equation}
P\bigl(|f-\text{\rm med}(f)|>t\bigr)\le 4\exp\left\{-\frac{t^2}{16L^2}\right\},
\end{equation}
where $\text{\rm med}(f)$ denotes the median of $f$.
\end{theorem}

\begin{proofsect}{Proof of Theorem~\ref{thm-1.2a}}
We will first prove concentration for the quantity $\Lambda_k^\epsilon(\xi)$ and then extract the desired statement from it.
In light of Theorem~\ref{thm-Talagrand}, it suffices to derive a good bound on the Lipschitz constant for~$f(\xi):=\Lambda_k^\epsilon(\xi)$. Fix~$\xi$ and let $\{g_{D_\epsilon,\xi}^{\ssup i}\colon i=1,\dots, k\}$ be a set of eigenfunctions satisfying \eqref{E:3.1} that achieve the corresponding eigenvalues $\{\lambda_{D_\epsilon,\xi}^{\ssup i}\colon i=1,\dots,k\}$, respectively. For any~$\eta$, the variational characterization \eqref{E:3.22} of~$\Lambda_k^\epsilon(\xi)$ yields
\begin{equation}
\label{E:3.23}
\Lambda_k^\epsilon(\xi)-\Lambda_k^\epsilon(\eta)\le\sum_{x\in D_\epsilon}\bigl(\xi(x)-\eta(x)\bigr)\sum_{j=1}^k\bigl|g_{D_\epsilon,{\eta}}^{\ssup j}(x)\bigr|^2.
\end{equation}
Peeling off the sum over~$j$ and applying the Cauchy-Schwarz inequality, we obtain
\begin{equation}
\label{E:3.24}
\Lambda_k^\epsilon(\xi)-\Lambda_k^\epsilon(\eta)\le
|\xi-\eta|_2\,\sum_{j=1}^k\,\biggl(\,\,\sum_{x\in D_\epsilon}\bigl|g_{D_\epsilon,{\eta}}^{\ssup j}(x)\bigr|^4\biggr)^{1/2}.
\end{equation}
But Lemma~\ref{prop4} ensures that $|g_{D_\epsilon,{\eta}}^{\ssup j}(x)|\le c\epsilon^{d/2}$, and the normalization convention \eqref{E:3.1} then gives
\begin{equation}
\label{E:4.5u}
\Lambda_k^\epsilon(\xi)-\Lambda_k^\epsilon(\eta)\le kc\epsilon^{d/2}\,|\xi-\eta|_2.
\end{equation}
Since this is valid for all~$\eta,\xi$, the same estimate applies to $|\Lambda_k^\epsilon(\xi)-\Lambda_k^\epsilon(\eta)|$ as well.

Now fix~$t>0$. Talagrand's inequality readily yields
\begin{equation}
\BbbP_\epsilon\bigl(|\Lambda_k^\epsilon-\text{med}(\Lambda_k^\epsilon)|>t\bigr)
\le {4} \exp\bigl\{-ct^2\epsilon^{-d}\bigr\}.
\end{equation}
But that implies the same bound also for
$\text{med}(\Lambda_k^\epsilon)$ replaced by $\E\Lambda_k^\epsilon$. Since $\Lambda_k^\epsilon(\xi)$ is the sum of the first~$k$ eigenvalues, the desired inequality for a single eigenvalue follows by considering the differences $\Lambda_k^\epsilon(\xi)-\Lambda_{k-1}^\epsilon(\xi)$.
\end{proofsect}

{\begin{remark}
We note that, thanks to pointwise boundedness of the support of~$\xi$ and the Lipschitz property of the eigenfunction, the proof could equally well be based on Azuma's inequality.
\end{remark}
}

For later purposes we restate the concentration bound in a slightly different form:

\begin{lemma}
\label{concentration}
Let~$k\ge1$. There is a constant $c>0$ such that for any
$t>0$,
\begin{equation}
\label{E:4.7eq}
\max_{x\in D_\epsilon}\,\BbbP_\epsilon\left(\sup_{\xi(x)\in[a,b]}
\bigl|\lambda_{D_\epsilon,\xi}^{\ssup k}-
\lambda_D^{\ssup k}\bigr|>t\right)
\le {4}\exp\left\{-c t^2\epsilon^{-d}\right\}
\end{equation}
holds for all sufficiently small $\epsilon$. 
\end{lemma}

\begin{proofsect}{Proof}
Let~$t>0$ be fixed. From Theorem~\ref{thm-1.2a} we know that $\lambda_{D_\epsilon,\xi}^{\ssup k}\to\lambda_D^{\ssup k}$ in probability. Since the eigenvalues are uniformly bounded, this implies
\begin{equation}
\bigl|\E\lambda_{D_\epsilon,\xi}^{\ssup k}-\lambda_D^{\ssup k}\bigr|<\frac23\,t
\end{equation}
for~$\epsilon>0$ sufficiently small. Moreover, \eqref{E:4.5u} gives 
\begin{equation}
\sup_{\begin{subarray}{c}
\xi(y)=\eta(y)\\\forall y\ne x
\end{subarray}}
\bigl|\lambda_{D_\epsilon,\xi}^{\ssup k}-\lambda_{D_\epsilon,\eta}^{\ssup k}\bigr|\le c\epsilon^{d/2}<\frac13 \,t,
\end{equation}
once~$\epsilon$ is sufficiently small. 
Hence, the probability in \eqref{E:4.7eq} is bounded by the probability that~$\lambda_{D_\epsilon,\xi}^{\ssup k}$ deviates from its mean by more than $t/3$. This is estimated using Theorem~\ref{thm-1.2a}.
\end{proofsect}

\section{Gaussian limit law}
\nopagebreak\label{sec5}\noindent
We are now finally ready to address the main aspect of this work, which is the limit theorem for fluctuations of asymptotically non-degenerate eigenvalues.
%
The main idea is quite simple and is inspired by the recent work on fluctuations of effective conductivity in the random conductance model (Biskup, Salvi and Wolff~\cite{BSW}). Consider an ordering of the vertices in~$D_\epsilon$ into a sequence $x_1,\dots,x_{|D_\epsilon|}$ and let $\FF_m:=\sigma(\xi({x_1}),\dots,\xi({x_m}))$. Then
\begin{equation}
\lambda_{D_\epsilon,\xi}^{\ssup{k}}-\E \lambda_{D_\epsilon,\xi}^{\ssup{k}}
=\sum_{m=1}^{|D_\epsilon|}
\Bigl(\,\E\bigl(\lambda_{D_\epsilon,\xi}^{\ssup{k}}\big|\FF_m\bigr)-\E\bigl(\lambda_{D_\epsilon,\xi}^{\ssup{k}}\big|\FF_{m-1}\bigr)\Bigr)
\end{equation}
represents the fluctuation of the $k$-th eigenvalue as a martingale. We may then apply the Martingale Central Limit Theorem due to Brown~\cite{Brown} which asserts that a family
\begin{equation}
\bigl\{(M_m^\epsilon,\FF_m)\colon m=0,\dots,n(\epsilon)\bigr\}
\end{equation}
of square-integrable $\R^\nu$-valued martingales such that
\begin{enumerate}
\item[(0)] $M_0^\epsilon=0$ and $n(\epsilon)\to\infty$ as $\epsilon\downarrow0$,
\item[(1)] there is a finite $\nu$-dimensional square matrix $\sigma^2=\{\sigma_{ij}^2\}$ for which
\begin{equation}
\epsilon^{-d}\sum_{m=1}^{n(\epsilon)}E\bigl((M_m^\epsilon-M_{m-1}^\epsilon)(M_m^\epsilon-M_{m-1}^\epsilon)^{\text{T}}\big|\FF_{m-1}\bigr)
\,\,\underset{\epsilon\downarrow0}{\overset{\BbbP}\longrightarrow}\,\,\sigma^2,
\end{equation}
\item[(2)] for each $\delta>0$,
\begin{equation}
\epsilon^{-d}\sum_{m=1}^{n(\epsilon)}E\bigl(|M_m^\epsilon-M_{m-1}^\epsilon|^2\1_{\{|M_m^\epsilon-M_{m-1}^\epsilon|>\delta\epsilon^{d/2}\}}\big|\FF_{m-1}\bigr)
\,\,\underset{\epsilon\downarrow0}{\overset{\BbbP}\longrightarrow}\,\,0,
\end{equation}
\end{enumerate}
satisfies
\begin{equation}
\epsilon^{-d/2}M_{n(\epsilon)}^\epsilon\,\,\underset{\epsilon\downarrow0}{\overset{\text{\rm law}}\longrightarrow}\,\,\NN(0,\sigma^2).
\end{equation}
The proof of Theorem~\ref{thm1.2} thus reduces to verification of the premises (0-2) of this result for 
\begin{equation}
M_n^\epsilon:=\sum_{m=1}^n
\Bigl(\,\E\bigl(\lambda_{D_\epsilon,\xi}^{\ssup{k}}\big|\FF_m\bigr)-\E\bigl(\lambda_{D_\epsilon,\xi}^{\ssup{k}}\big|\FF_{m-1}\bigr)\Bigr)
\end{equation}
and $n(\epsilon):=|D_\epsilon|$.

The condition (0) is checked immediately, but the control of the limits in (1) and (2) will require a more explicit expression for the martingale differences. Here we note that, for any function $f=f(\xi_1,\dots,\xi_n)$ on~$\R^{n}$ that is absolutely continuous in each variable and for any collection $\xi_1,\dots,\xi_n$ of bounded independent random variables we have, {for} $\FF_m:=\sigma(\xi_1,\dots,\xi_m)$,
\begin{equation}
\label{E:2.2}
\E(f|\FF_m)-\E(f|\FF_{m-1})=\widehat\E\int_{\widehat\xi_m}^{\xi_m}\frac{\partial f}{\partial\xi_m}(\xi_1,\dots,\xi_{m-1},\tilde\xi,\widehat\xi_{m+1},\dots,\widehat\xi_n)\,\textd\tilde\xi,
\end{equation}
where the expectation is over the collection of random variables $\widehat\xi$, which are copies of~$\xi$ independent of~$\xi$. The integral is in the sense of Riemann, and we use the corresponding notation to explicate the sign change upon exchanging the limits of integration. To validate the condition of absolute continuity (and justify the use of the Fundamental Theorem of Calculus), we prove:

\begin{lemma}
\label{lemma-5.1}
The function $\xi\mapsto\lambda_{D_\epsilon,\xi}^{\ssup{k}}$ is everywhere right and left differentiable {with respect to} each~$\xi(x)$. The set of points where the two derivatives disagree is at most countably infinite; else the derivative exists and is continuous in~$\xi(x)$. The partial derivatives $\frac\partial{\partial\xi(x)^\pm}\lambda_{D_\epsilon,\xi}^{\ssup{k}}$ are bounded and, except at countably many values of~$\xi(x)$,
\begin{equation}
\label{E:2.3}
\frac\partial{\partial\xi(x)}\lambda_{D_\epsilon,\xi}^{\ssup{k}}=\bigl|g_{D_\epsilon,\xi}^{\ssup k}(x)\bigr|^2
\end{equation}
for any possible choice of $g_{D_\epsilon,\xi}^{\ssup k}$. (I.e., all choices give the same result.) 
\end{lemma}

\begin{proofsect}{Proof}
Note that $\lambda_{D_\epsilon,\xi}^{\ssup{k}}=\Lambda_k^\epsilon(\xi)-\Lambda_{k-1}^\epsilon(\xi)$. Since $\xi\mapsto\Lambda_k^\epsilon(\xi)$ is concave --- being the infimum of a family of linear functions --- it is right and left differentiable in~$\xi(x)$ at all values. The derivatives are non-increasing and ordered so there are at most countably many points where they disagree. Moreover, at differentiability points of $\Lambda_k^\epsilon$, \eqref{E:3.23} yields
\begin{equation}
{\frac\partial{\partial\xi(x)}}\Lambda_k^\epsilon(\xi)=\sum_{j=1}^k\bigl|g_{D_\epsilon,\xi}^{\ssup j}(x)\bigr|^2
\end{equation}
for any choice of eigenfunctions $g_{D_\epsilon,\xi}^{\ssup 1},\dots,g_{D_\epsilon,\xi}^{\ssup k}$. At common differentiability points of both~$\Lambda_k^\epsilon(\xi)$ and $\Lambda_{k-1}^\epsilon(\xi)$, we then get \eqref{E:2.3}.
\end{proofsect}

The upshot of Lemma~\ref{lemma-5.1} is that we are permitted to use \eqref{E:2.3} in \eqref{E:2.2} with no provisos on eigenvalue degeneracy. Our goal is to replace the modulus-squared of $g_{D_\epsilon,\xi}^{\ssup k}$ by that pertaining to the corresponding eigenfunction in the continuum problem. However, there is a subtle issue arising from the integration with respect to the dummy variable~$\tilde\xi$ in \eqref{E:2.2}. Indeed, with this variable in  place of~$\xi(x)$, the configuration~$\xi$ may not even be in the support of~$\BbbP_\epsilon$. We handle this with the help of:

\begin{lemma}
\label{prop3}
Given~$k\ge1$ and a configuration $\xi$, suppose that $\lambda_{D_\epsilon,\xi}^{\ssup k}$ remains simple as~$\xi(x)$ varies through an interval~$[a,b]$. Then for any $\xi'$ satisfying $\xi(y)=\xi'(y)$ for~$y\ne x$ and for any $\xi(x),\xi'(x)\in[a,b]$,
\begin{equation}
\label{E:g-g}
\bigl|g_{D_{\epsilon},\xi'}^{\ssup k}(x)\bigr| = \bigl|g_{D_{\epsilon},\xi}^{\ssup k}(x)\bigr|
\exp\biggl\{\int_{\xi(x)}^{\xi'(x)}G_{D_\epsilon}^{\ssup k}(x,x;\tilde\xi)\,\textd\tilde\xi(x)\biggr\},
\end{equation}
where~$\tilde\xi$ is the configuration that agrees with $\xi$ (and~$\xi'$) outside~$x$ where it equals~$\tilde\xi(x)$ and
\begin{equation}
G_{D_\epsilon}^{\ssup k}(x,y;\xi):=\bigl\langle\delta_x,(H_{D_\epsilon,\xi}-\lambda_{D_\epsilon,\xi}^{\ssup k})^{-1}(1-\widehat P_k)\delta_y\bigr\rangle_{\ell^2(\Z^d)}
\end{equation}
with~$\widehat P_k$ denoting the orthogonal projection on $\text{\rm Ker}(\lambda_{D_\epsilon,\xi}^{\ssup k}-H_{D_\epsilon,\xi})$ and $H_{D_\epsilon,\xi}-\lambda_{D_\epsilon,\xi}^{\ssup k}$ now regarded as an operator acting on $\text{\rm Ker}(\lambda_{D_\epsilon,\xi}^{\ssup k}-H_{D_\epsilon,\xi})^{\perp}$.
\end{lemma}

\begin{proofsect}{Proof}
To make notations brief, let us write~$\lambda$, resp.,~$g$ for the relevant eigenvalue, resp., eigenfunction. Since the eigenvalue is simple, Rayleigh's perturbation theory ensures that the eigenfunction is unique up to normalization and overall sign. In particular, \eqref{E:2.3} holds. Moreover, also the eigenfunction~$g$ --- with the sign fixed at~$x$, for instance --- is differentiable in~$\xi(x)$. Taking the derivative of the eigenvalue equation, we get
\begin{equation}
(\lambda-H_{D_\epsilon,\xi})\frac{\partial g}{\partial\xi(x)}=g(x){1_{\{x\}}}-|g(x)|^2 \,g.
\end{equation}
Note that we have $\langle g,\frac{\partial g}{\partial\xi(x)}\rangle=0$ by differentiating $\|g\|_2=1$. Interpreting the right-hand side as $(1-\widehat P_k){(g(x)1_{\{x\}})}$, we can now invert {$\lambda-H_{D_{\epsilon},\xi}$} to obtain
\begin{equation}
\frac{\partial}{\partial\xi(x)}g(y)=G_{D_\epsilon}^{\ssup k}(y,x;\xi)g(x).
\end{equation}
Evaluating at~$x$, we get an autonomous ODE for~$g(x)$. Solving yields \eqref{E:g-g}.
\end{proofsect}

Our next aim will be to show that, whenever $\lambda_D^{\ssup k}$ is simple, the term in the exponent of \eqref{E:g-g} actually tends to zero as~$\epsilon\downarrow0$.

\begin{lemma}
\label{unif-g}
For~$k\ge1$ let $\delta$ be such that $0<\delta<\frac13\min\{\lambda_D^{\ssup k}-\lambda_D^{\ssup{k-1}},\lambda_D^{\ssup{k+1}}-\lambda_D^{\ssup k}\}$ and set
\begin{equation}
A_{k,\epsilon}:=\bigcap_{x\in D_\epsilon}\Bigl\{\xi\colon\sup_{\xi_x\in[a,b]}|\lambda_{D_\epsilon,\xi}^{\ssup i}-\lambda_D^{\ssup i}|<\delta,\, i=k-1,k,k+1\Bigr\}.
\end{equation}
Then
\begin{equation}
\max_{x\in D_\epsilon}\,\sup_{\xi_x'\in[a,b]}\,\,\sup_{\xi\in A_{k,\epsilon}}\,\,\Bigl|\int_{\xi_x}^{\xi_x'}G^{\ssup k}_{D_\epsilon}(x,x;\tilde\xi)\textd\tilde\xi_x\Bigr|\,\,\underset{\epsilon\downarrow0}\longrightarrow\,\,0.
\end{equation}
\end{lemma}

\begin{proofsect}{Proof}
Take~$k$ such that $\lambda_D^{\ssup k}$ is simple and note that, for~$\xi\in A_{k,\epsilon}$, the eigenvalue~$\lambda_{D_\epsilon,\xi}^{\ssup k}$ remains simple for all values of~$\xi(x)$. Then
\begin{equation}
\label{E:5.15}
G_{D_\epsilon}^{\ssup k}(x,x;\xi)=\sum_{\begin{subarray}{c}
i\ge1\\i\ne k
\end{subarray}}
\,\frac1{\lambda_{D_\epsilon,\xi}^{\ssup i}-\lambda_{D_\epsilon,\xi}^{\ssup k}}\,\bigl|g_{D_\epsilon,\xi}^{\ssup i}(x)\bigr|^2.
\end{equation}
Thanks to \eqref{E:3.4aa} and the fact that the eigenvalues of $-\epsilon^{-2}\Delta^{(\textd)}$ are close to those of the continuum problem,  for each~$R>0$ there is $K>k$ such that, for any sufficiently small $\epsilon>0$,
\begin{equation}
i\ge K\quad\Rightarrow\quad \lambda_{D_\epsilon,\xi}^{\ssup i}\ge \lambda_{D_\epsilon,\xi}^{\ssup k}+R
\end{equation}
uniformly in~$\xi\in \Omega_{a,b}$. The corresponding part of the above sum is then bounded by
\begin{equation}
0\le\sum_{i\ge K}\,\frac1{\lambda_{D_\epsilon,\xi}^{\ssup i}-\lambda_{D_\epsilon,\xi}^{\ssup k}}\,\bigl|g_{D_\epsilon,\xi}^{\ssup i}(x)\bigr|^2
\le\frac1R\sum_{i\ge K}\bigl|g_{D_\epsilon,\xi}^{\ssup i}(x)\bigr|^2\le\frac1R,
\end{equation}
where we used the Plancherel formula to bound the second sum by $\langle\delta_x,\delta_x\rangle_{2}=1$. This reduces an estimate of~$G_{D_\epsilon}^{\ssup k}(x,x;\xi)$ to a finite number of terms.

On~$A_{k,\epsilon}$ \eqref{E:2.3} and Lemma~\ref{prop4} show, for all~$\epsilon$ sufficiently small,
\begin{equation}
\forall\xi\in A_{k,\epsilon}\colon\qquad\sup_{\xi(x)\in[a,b]}
\bigl|\lambda_{D_\epsilon,\xi}^{\ssup i}-\lambda_{D_\epsilon,\xi}^{\ssup k}\bigr|>\frac\delta3-c\epsilon^{d}|b-a|>\frac\delta4,\qquad i=k-1,k+1.
\end{equation} 
The sum of first~$K$ terms in \eqref{E:5.15} can thus be bounded by $cK\delta^{-1}\epsilon^d$, uniformly on~$A_{k,\epsilon}$. This permits us to take~$R\to\infty$ simultaneously with~$\epsilon\downarrow0$ and conclude the claim.
\end{proofsect}

Given~$\epsilon>0$, consider now an ordering $x_1,\dots,x_{|D_\epsilon|}$ of vertices of~$D_\epsilon$ and given $\xi,\widehat\xi\in\Omega_{a,b}$, denote by $\widehat\xi^{(m)}$ the configuration
\begin{equation}
\widehat\xi^{(m)}(x_i):=\begin{cases}
\xi(x_i),\qquad&\text{if }i\le m,
\\
\widehat\xi(x_i),\qquad&\text{if }i> m.
\end{cases}
\end{equation}
Hereafter, we regard $\widehat{\xi}$ as {an independent copy of
$\xi$} and denote the corresponding expectation by~$\widehat{\E}$. 
Let $\FF_m:=\sigma(\xi(x_1),\dots,\xi(x_m))$. 
The martingale difference can then be written with the help of Lemma~\ref{lemma-5.1} as 
\begin{equation}
\label{E:5.20}
\begin{aligned}
Z_m^{\ssup i}&:=\E\bigl(\lambda_{D_\epsilon,\xi}^{{\ssup{i}}}\big|\FF_{m}\bigr)
 -\E\bigl(\lambda_{D_\epsilon,\xi}^{{\ssup{i}}}\big|\FF_{m-1}\bigr)
 \\*[2mm]
 &\phantom{:}=\,\widehat\E\Bigl(\lambda_{D_\epsilon,\widehat\xi^{(m)}}^{{\ssup{i}}}
-\lambda_{D_\epsilon,\widehat\xi^{(m-1)}}^{{\ssup{i}}}\Bigr)
\\*[1mm]
&\phantom{:}=\widehat\E\biggl(\int_{\widehat\xi(x_m)}^{\xi(x_m)}\bigl|g_{D_\epsilon,\wt\xi^{(m)}}^{\ssup i}({x_m})\bigr|^2\textd\tilde\xi\biggr),
\end{aligned}
\end{equation}
where $\wt\xi^{(m)}$ is the configuration that equals $\xi$ on $\{x_1,\dots,x_{m-1}\}$, takes value $\tilde\xi$ at~$x_m$, and coincides with~$\widehat\xi$ on $\{x_{m+1},\dots,x_{|D_\epsilon|}\}$. Notice that Lemma~\ref{prop4}  immediately gives
\begin{equation}
\label{E:5.21}
|Z_m^{\ssup i}|\le c\epsilon^d
\end{equation}
for some constant~$c<\infty$. In particular, condition~(2) in the abovementioned Martingale Central Limit Theorem holds trivially. For condition~(1), we will proceed, as mentioned before, by replacing the square of the discrete eigenfunction by its corresponding continuum counterpart. The key estimate is stated in:

\begin{proposition}
\label{prop-5.4}
Suppose~$\lambda_D^{{\ssup{i}}}$ and $\lambda_D^{{\ssup{j}}}$ are simple. Abbreviate $B_\epsilon(x):=\epsilon x+[0,\epsilon)^d$. Then we have:
\begin{equation}
\E\,\Biggl|
\,\sum_{m=1}^{|D_\epsilon|}\biggl( \E\bigl(
 (\epsilon^{-d} Z_m^{\ssup i})(\epsilon^{-d} Z_m^{\ssup j})\,\big|\,\FF_{m-1}\bigr)-\int_{B_\epsilon(x_m)}\!\!\textd y\,\,V(y)\bigl|\varphi_D^{{\ssup{i}}}(y)\bigr|^2\bigl|\varphi_D^{{\ssup{j}}}(y)\bigr|^2\biggr)\Biggr|\,\underset{\epsilon\downarrow0}\longrightarrow\,0.
\end{equation}
\end{proposition}

The proof of this proposition will be done in several steps. Recall the definition of event $A_{k,\epsilon}$ and note that,
on~$A_{k,\epsilon}$ the eigenfunction $g_{D_\epsilon,\xi}^{\ssup k}$ is unique up to a sign and, in particular, there is a unique measurable version of~$\xi\mapsto |g_{D_\epsilon,\xi}^{\ssup k}(x)|^2$ for each~$x$. In light of the concentration bound in Lemma~\ref{concentration} we have
\begin{equation}
\label{E:5.23}
\lambda_D^{\ssup k}\text{ simple}\quad\Rightarrow\quad
\BbbP_\epsilon\bigl(A_{k,\epsilon}\bigr)\,\underset{\epsilon\downarrow0}\longrightarrow\,1.
\end{equation}
Our first replacement step is the content of:

\begin{lemma}
\label{lemma-5.5}
Suppose $\lambda_D^{\ssup k}$ is simple. Then
\begin{equation}
\label{E:5.25}
\epsilon^{-d}\sum_{m=1}^{|D_\epsilon|}
\,
\E
\Biggl(\,\biggl| \, Z_m^{\ssup k}-\bigl(\xi(x_m)-U(\epsilon x_m)\bigr)\E\Bigl(\,\bigl|g_{D_\epsilon,\xi}^{\ssup {k}}(x_m)\bigr|^2
\1_{A_{k,\epsilon}}\,\Big|\,\FF_m\Bigr)\biggr|^2\Biggr)\,\underset{\epsilon\downarrow0}\longrightarrow\,0.
\end{equation}
\end{lemma}

\begin{proofsect}{Proof}
Inserting the indicator of $\widehat\xi^{(m)}\in A_{k,\epsilon}$ and/or its complement into the third line of \eqref{E:5.20} and applying the boundedness of the discrete eigenfunctions from Lemma~\ref{prop4} shows
\begin{equation}
\label{E:5.25a}
\Biggl|\,Z_m^{\ssup k}-\widehat\E\biggl(\1_{\{\widehat\xi^{(m)}\in A_{k,\epsilon}\}}\int_{\widehat\xi(x_m)}^{\xi(x_m)}\bigl|g_{D_\epsilon,\wt\xi^{(m)}}^{\ssup k}({x_m})\bigr|^2\textd\tilde\xi\biggr)\Biggr|
\le c\epsilon^d\,\E(\1_{A_{k,\epsilon}^\cc}|\FF_{m}),
\end{equation}
where, we recall, the expectation~$\widehat\E$ affects only~$\widehat\xi$ and so $\widehat\E(\1_{\{\widehat\xi^{(m)}\not\in A_{k,\epsilon}\}})=\E(\1_{A_{k,\epsilon}^\cc}|\FF_{m})$. {Abbreviate temporarily
\begin{equation}
F_m(\tilde\xi):=\exp\biggl\{2\int_{\xi(x_m)}^{\tilde\xi}G_{D_\epsilon}^{\ssup{k}}({x_m,x_m};\tilde\xi')\textd\tilde\xi'\biggr\}.
\end{equation}
On the event $\{\widehat\xi^{(m)}\in A_{k,\epsilon}\}$, Lemmas~\ref{lemma-5.1} and~\ref{prop3} along with $\widehat\xi^{(m)}(x_m)=\xi(x_m)$ yield
\begin{equation}
\begin{aligned}
\int_{\widehat\xi(x_m)}^{\xi(x_m)}\bigl|g_{D_\epsilon,\wt\xi^{(m)}}^{\ssup k}&(x_m)\bigr|^2\textd\tilde\xi
\,-\,\bigl({\xi(x_m)}-{\widehat\xi(x_m)}\bigr)\bigl|g_{D_\epsilon,\widehat\xi^{(m)}}^{\ssup k}(x_m)\bigr|^2
\\
&=\int_{\widehat\xi(x_m)}^{\xi(x_m)}\Bigl(\bigl|g_{D_\epsilon,\wt\xi^{(m)}}^{\ssup k}(x_m)\bigr|^2
\,-\,\bigl|g_{D_\epsilon,\widehat\xi^{(m)}}^{\ssup k}(x_m)\bigr|^2\Bigr)\textd\tilde\xi
\\
&\qquad\qquad=\bigl|g_{D_\epsilon,\widehat\xi^{(m)}}^{\ssup k}(x_m)\bigr|^2\,
\int_{\widehat\xi(x_m)}^{\xi(x_m)}\bigl(F_m(\tilde\xi)-1\bigr)\textd\tilde\xi.
\end{aligned}
\end{equation}
Lemma~\ref{unif-g} then bounds the difference $F_m(\tilde\xi)-1$ uniformly by $\texte^{\delta(\epsilon)}-1$ for some
$\delta(\epsilon)>0$ that tends to zero as~$\epsilon\downarrow0$. 
Thanks to the uniform boundedness of the eigenfunctions, this and \eqref{E:5.25a} yield
\begin{multline}
\qquad
\Biggl|\,Z_m^{\ssup k}-\widehat\E\biggl(\1_{\{\widehat\xi^{(m)}\in A_{k,\epsilon}\}}\bigl({\xi(x_m)}-{\widehat\xi(x_m)}\bigr)\bigl|g_{D_\epsilon,\widehat\xi^{(m)}}^{\ssup k}(x_m)\bigr|^2\biggr)\Biggr|
\\
\le c\epsilon^d\,\E(\1_{A_{k,\epsilon}^\cc}|\FF_{m})+c\epsilon^d\bigl(\texte^{\delta(\epsilon)}-1\bigr).
\qquad
\end{multline}
The configuration $\widehat\xi^{(m)}$ does not depend on~$\widehat\xi(x_m)$, and so we may take expectation with respect to~$\widehat\xi(x_m)$ and effectively replace it by $U(\epsilon x)$.
Recasting~$\widehat\E$ as conditional expectation given~$\FF_m$ and using that~$\xi(x_m)$ is $\FF_m$-measurable, we thus conclude
\begin{multline}
\qquad
\biggl| \, Z_m^{\ssup k}-
\bigl(\xi(x_m)-U(\epsilon x_m)\bigr)
\E\Bigl(\,\bigl|g_{D_\epsilon,\xi}^{\ssup {k}}({x_m})\bigr|^2
\1_{A_{k,\epsilon}}\,\Big|\,\FF_m\Bigr)\biggr|
\\
\le c\epsilon^d\,\E(\1_{A_{k,\epsilon}^\cc}|\FF_{m})+c\epsilon^d\bigl(\texte^{\delta(\epsilon)}-1\bigr).
\qquad\end{multline}
Squaring this and taking another expectation shows that the left-hand-side of \eqref{E:5.25} is bounded by $c\epsilon^{d} |D_\epsilon|$ times $\BbbP_\epsilon(A_{k,\epsilon}^\cc)+(\texte^{\delta(\epsilon)}-1)$. By \eqref{E:5.23}, this tends to zero as claimed.
}
\end{proofsect}

Next we note:

\begin{lemma}
\label{lemma-5.6}
Suppose $\lambda_D^{\ssup k}$ is simple. Then
\begin{equation}
\label{E:5.28aa}
\sum_{m=1}^{|D_\epsilon|}\int_{B_\epsilon(x_m)}\textd y\,\,
\E\biggl(\,\Bigr|\,\bigl|\varphi_D^{\ssup{k}}(y)\bigr|^2-\epsilon^{-d}\bigl|g_{D_\epsilon,\xi}^{\ssup {k}}({x_m})\bigr|^2\1_{A_{k,\epsilon}}\Bigr|\biggr)\,\underset{\epsilon\downarrow0}\longrightarrow\,0
\end{equation}
\end{lemma}

\begin{proofsect}{Proof}
Recall the setting of Corollary~\ref{cor3.8} and, in particular, given the scaled discrete eigenfunctions $\epsilon^{-d/2}g_{D_\epsilon,\xi}^{\ssup 1},\dots,\epsilon^{-d/2}g_{D_\epsilon,\xi}^{\ssup k}$, let $\wt g_{1,\xi}^{\,\epsilon},\dots,\wt g_{k,\xi}^{\,\epsilon}$ denote their continuum interpolations. As~$\lambda_D^{\ssup k}$ is simple, Corollary~\ref{cor3.8} guarantees that these functions project almost entirely onto the closed linear span of $\{\varphi_D^{\ssup1},\dots,\varphi_D^{\ssup \ell}\}$ for both~$\ell=k-1$ and~$\ell=k$. As these functions are also nearly orthogonal, we get
\begin{equation}
\BbbP_\epsilon\Bigl(\,A_{k,\epsilon}\,\,\&\,\,\bigl\Vert \,|\wt g_{k,\xi}^{\,\epsilon}|-|\varphi_D^{\ssup k}|\,\bigr\Vert_{L^2(D)}>\delta\Bigr)
\,\underset{\epsilon\downarrow0}\longrightarrow\,0
\end{equation}
for any~$\delta>0$. As both $|\wt g_{k,\xi}^{\,\epsilon}|$ and $|\varphi_D^{\ssup k}|$ are uniformly bounded, this implies
\begin{equation}
\label{E:5.31}
\int_{\R^d}\textd y\,\,
\E\biggl(\,\Bigr|\,\bigl|\varphi_D^{\ssup{k}}(y)\bigr|^2-\bigl|\wt g_{k,\xi}^{\,\epsilon}(y)\bigr|^2\1_{A_{k,\epsilon}}\Bigr|\biggr)\,\underset{\epsilon\downarrow0}\longrightarrow\,0
\end{equation}
with the help of {\eqref{E:5.23}}. But \eqref{E:3.14uu} gives
\begin{equation}
\label{E:5.32}
\sum_{m=1}^{|D_\epsilon|}\int_{B_\epsilon(x_m)}\textd y\,\,
\E\biggl(\,\Bigr|\,\wt g_{k,\xi}^{\,\epsilon}(y)-\epsilon^{-d/2}g_{D_\epsilon,\xi}^{\ssup {k}}(x_m)\Bigr|^2\1_{A_{k,\epsilon}}\biggr)\le C(d)\E\Bigl(\Vert\nabla^{(\textd)}g_{D_\epsilon,\xi}^{\ssup {k}}\Vert_2^2\1_{A_{k,\epsilon}}\Bigr),
\end{equation}
which tends to zero proportionally to~$\epsilon^2$, due to boundedness of the kinetic energy. Combining \twoeqref{E:5.31}{E:5.32}, we get the claim.
\end{proofsect}

\begin{proofsect}{Proof of Proposition~\ref{prop-5.4}}
Combining Lemmas~\ref{lemma-5.5} and~\ref{lemma-5.6}, and using
{that} the conditional expectation is a contraction in~$L^2$, we get
\begin{equation}
\label{E:5.33}
\sum_{m=1}^{|D_\epsilon|}\int_{B_\epsilon(x_m)}\textd y\,\,
\,\E
\biggl(\,\Bigl| \, \epsilon^{-d}Z_m^{\ssup k}-\bigl(\xi(x_m)-U(\epsilon x_m)\bigr)\bigl|\varphi_D^{\ssup{k}}(y)\bigr|^2\Bigr|^2\biggr)\,\underset{\epsilon\downarrow0}\longrightarrow\,0.
\end{equation}
for both $k={i,j}$. The claim now reduces to
\begin{equation}
\sum_{m=1}^{|D_\epsilon|}\int_{B_\epsilon(x_m)}\textd y\,\bigl|V(y)-V(\epsilon x_m)\bigr|\,\bigl|\varphi_D^{{\ssup{i}}}(y)\bigr|^2\bigl|\varphi_D^{{\ssup{j}}}(y)\bigr|^2\,\underset{\epsilon\downarrow0}\longrightarrow\,0,
\end{equation}
which follows by uniform continuity of $y\mapsto V(y)$ and the boundedness of the eigenfunctions.
\end{proofsect}

\begin{proofsect}{Proof of Theorem~\ref{thm1.2}}
Thanks to Proposition~\ref{prop-5.4} and the fact that $|B_\epsilon({x_m})|=\epsilon^d$,
\begin{equation}
\epsilon^{-d}\sum_{m=1}^{|D_\epsilon|}\E\bigl(Z_m^{{\ssup{k_i}}}
Z_m^{{\ssup{k_j}}}\big|\FF_{m-1}\bigr)
\,\underset{\epsilon\downarrow0}\longrightarrow\,
\int_{D}V(y)\bigl|\varphi_D^{\ssup{k_i}}(y)\bigr|^2\bigl|\varphi_D^{\ssup{k_j}}(y)\bigr|^2\,\textd y
\end{equation}
in $L^1(\BbbP_\epsilon)$ and thus in probability. This verifies the (last yet unproved) condition~(1) of the Martingale Central Limit Theorem and so the result follows.
\end{proofsect}

\begin{proofsect}{Proof of Theorem~\ref{thm1.3}}
The relation \eqref{E:1.11} is a direct consequence of 
Lemma~\ref{lemma-5.6} {and the boundedness of eigenfunctions}.
For \eqref{E:1.10} we again drop the suffixes on all quantities and write, on~$A_{k,\epsilon}$,
\begin{multline}
\qquad
T^{\ssup k}-\E(T^{\ssup k}\1_{A_{k,\epsilon}})
=\lambda^{\ssup k}-\E(\lambda^{\ssup k}\1_{A_{k,\epsilon}})-\sum_{x\in D_\epsilon}\bigl(\xi(x)-U(x\epsilon)\bigr)\bigl|g^\ssup k(x)\bigr|^2
\\
+\sum_{x\in D_\epsilon}
\Bigl(U(x\epsilon)\bigl|g^\ssup k(x)\bigr|^2-\E\bigl(\xi(x)\bigl|g^\ssup k(x)\bigr|^2\1_{A_{k,\epsilon}}\bigr)\Bigr)
\qquad
\end{multline}
Lemma~\ref{lemma-5.6} {and the boundedness of eigenfunctions} now allows us to replace the square of the discrete eigenfunction by $\epsilon^d|\varphi_D^{\ssup k}{(x\epsilon)}|^2$ up to an error that is negligible at overall scale~$\epsilon^d$. Using $Z_m:=\E(\lambda^{\ssup k}|\FF_m)-\E(\lambda^{\ssup k}|\FF_{m-1})$, we thus~get
\begin{equation}
\epsilon^{-d/2}\bigl(T^{\ssup k}-\E(T^{\ssup k}\1_{A_{k,\epsilon}})\bigr)
=o(1)+\sum_{m=1}^{{|D_\epsilon|}}\Bigl(\epsilon^{-d}Z_m-\bigl(\xi(x_m)-U(\epsilon x_m)\bigr)\bigl|\varphi_D^{\ssup k}(\epsilon x_m)\bigr|^2\Bigr),
\end{equation}
{where $o(1)$ represents a random variable whose variance vanishes
as $\epsilon$ goes to zero.}
The sum on the right is a martingale and so its variance is estimated by sum of variances of individual terms. Using a slight modification of \eqref{E:5.33}, the result tends to zero as~$\epsilon\downarrow0$. 
\end{proofsect}

\section*{Acknowledgments}
\noindent
This research has been partially supported by DFG Forschergruppe 718
``Analysis and Stochastics in Complex Physical Systems,'' NSF grant
DMS-1106850, NSA grant H98230-11-1-0171, GA\v CR project
P201-11-1558, JSPS KAKENHI Grant Number 24740055 and JSPS and DFG under
the
Japan-Germany Research Cooperative Program. This project was begun when M.B.\ was a long-term visitor of RIMS at Kyoto University, whose hospitality is gratefully acknowledged. M.B.\ thanks Yu Gu for illuminating discussions on various points in this article.

\end{document}